\documentclass[11pt]{amsart}
\usepackage{geometry}                
\usepackage{graphicx}
\usepackage{amssymb}
\usepackage{hyperref}
\theoremstyle{plain} 
\newtheorem{theorem}{Theorem}[section]
\newtheorem{lemma}[theorem]{Lemma}

\newtheorem{proposition}[theorem]{Proposition}
\theoremstyle{definition} 
\newtheorem{definition}[theorem]{Definition}
\newtheorem{remark}[theorem]{Remark}

\newtheorem{model}[theorem]{Model}
\newcommand{\BC}{\mathbb{C}}
\newcommand{\BE}{\mathbb{E}}
\newcommand{\BH}{\mathbb{H}}
\newcommand{\BR}{\mathbb{R}}

\newcommand{\SF}{\mathcal{F}}
\newcommand{\SH}{\mathcal{H}}
\newcommand{\SM}{\mathcal{M}}
\newcommand{\SR}{\mathcal{R}}

\newcommand{\SY}{\mathcal{Y}}

\newcommand{\re}{\operatorname{Re}}
\newcommand{\im}{\operatorname{Im}}
\newcommand{\dev}{\operatorname{dev}}
\newcommand{\per}{\operatorname{Per}}
\newcommand{\ext}[0]{\ensuremath{\mathrm{ext}}}
\numberwithin{equation}{section}
\numberwithin{figure}{section}

\def\tec{Teich\-m\"ul\-ler\ }
\def\wei{Weierstrass\ }

\def\p{\partial}
\def\ov{\overline}
\def\om{\omega}
\def\cd{\cdot}
\def\f{\frac}
\def\g{\gamma}
\def\lra{\longrightarrow}
\def\vp{\varphi}
\def\e{\epsilon}
\def\z{\zeta}
\def\oper{\operatorname}
\def\G{\Gamma}
\def\ogup{\Omega_{Gdh}}
\def\ogdn{\Omega_{G^{-1}dh}}
\def\ga{Gauss\ }
\def\height{\operatorname{\SH}}
\def\res{\oper{Res}}

\def\supp{\oper{supp}}
\def\id{\oper{id}}
\def\k{\kappa}
\def\lm{\limits}
\def\x{\times}

\def\wh{\widehat}

\begin{document}

\title{Handle Addition for doubly-periodic Scherk Surfaces}

 \author[Matthias Weber]{Matthias Weber}
   \address{Indiana University}
   \thanks{The first author was partially supported by
    NSF grant DMS-0139476.}
   \email{matweber@indiana.edu}
    \urladdr{http://www.indiana.edu/~minimal}
   \author[M. Wolf]{Michael Wolf}
   \address{Rice University}
   \thanks{The second author was partially supported by NSF grants DMS-9971563 and DMS-0139887.}
   \email{mwolf@math.rice.edu}
   \urladdr{http://www.math.rice.edu/\textasciitilde mwolf/}


\begin{abstract}
We prove the existence of a family of embedded doubly periodic minimal
surfaces of (quotient) genus $g$ with orthogonal ends that generalizes
the classical doubly periodic surface of Scherk and the genus-one
Scherk surface of Karcher.  The proof of the family of immersed
surfaces is by induction on genus, 
while the proof of embeddedness is by the conjugate Plateau method. 
\end{abstract}


\maketitle

\section{Introduction}
\label{sec1}

In this note we prove the existence of
a sequence $\{S_g\}$ of embedded doubly-periodic minimal surfaces,
beginning with the classical Scherk surface, indexed by the
number $g$ of handles in a fundamental domain. 
Formally, we prove
\begin{theorem}
There exists a family $\{S_g\}$ of embedded minimal surfaces,
invariant under a rank two group $\Lambda_g$ generated by horizontal orthogonal
translations. The quotient of each surface $S_g$ by $\Lambda_g$
has genus $g$ and four vertical ends arranged into two orthogonal
pairs.
\end{theorem}

Our interest in these surfaces has a number of sources. First, of
course, is that these are a new family of embedded doubly periodic
minimal surfaces with high topological complexity but relatively small
symmetry group for their quotient genus. Next, unlike the surfaces 
produced through desingularization of degenerate configurations (see
\cite{tra1}, \cite{tra7} for example), these surfaces are not 
created as members of a degenerating family or are even known to be
close to a degenerate surface.
More concretely, there is now an abundance of embedded doubly periodic minimal 
surfaces with parallel ends due to \cite{cowe1}, while in the case of non-parallel ends, the
Scherk and Karcher-Scherk surfaces were the only examples.

 Third, one can imagine these surfaces
as the initial point for a sheared family of (quotient) genus $g$ embedded
surfaces that would limit to a translation-invariant (quotient) 
genus $g$ helicoid: such a program has recently been implemented for
case of genus one
by Baginsky-Batista \cite{brb1} and Douglas \cite{dou1}.

Our final reason is that there is a novelty to our argument in this
paper in that we combine \wei representation techniques for creating
immersed minimal surfaces of arbitrary genus with conjugate Plateau 
methods for producing embedded surfaces. The result is then
embedded surfaces of arbitrary (quotient) genus.

Intuitively,
our method to create the family 
of immersed surfaces --- afterwards proven embedded --- is 
to  add a handle within a
fundamental domain, and then flow within a moduli space of
such surfaces to a minimal representative. We developed the
 method of proof in \cite{ww1} and \cite{ww2} 
of using the theory of flat structures
to add handles
to the classical Enneper's surface and the semi-classical Costa
surface; here we observe that the method easily extends to
the case of the doubly-periodic Scherk surface --- indeed, we
will compute that the relevant flat structures for Scherk's
surface with handles are close cousins to the relevant flat
structures for Enneper's surface with handles.
(This is a small surprise as the two surfaces are not usually
regarded as having similar geometries.)

Finally, we look at a fundamental domain on the surface for the 
automorphism group of the surface and analyze its conjugate surface.
As this turns out to be a graph, Krust's theorem implies that our
original fundamental domain is embedded.

\begin{figure}[h] 
\centering
\includegraphics[width=2in]{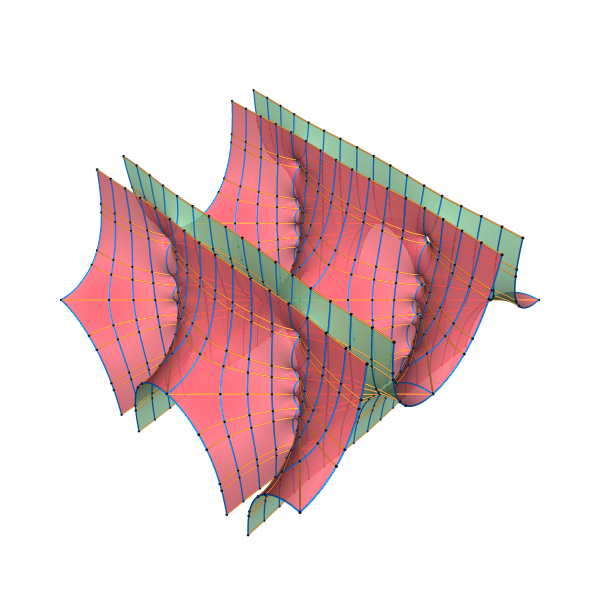} 
\caption{Scherk's surface with four additional handles}
\label{fig:genus4}
\end{figure}

Our paper is organized as follows: in the second section, we
recall the background information about the \wei
representation, conjugate surfaces, and \tec Theory, 
which we will need to construct our family of
surfaces. In the third section, we outline our method and
begin the construction by computing triples of relevant flat
structures corresponding to candidate for the \wei
representation for the $g$-handled Scherk surfaces. In the
fourth section, we define a finite dimensional moduli space
$\SM_g$ of such triples and define a non-negative height
function $\height:\SM_g\to\BR$ on that moduli space; a well-defined
$g$-handled Scherk surface $S_g$ will correspond to a zero of
that height function. Also in section 4, we prove that this height
function is proper on $\SM_g$.

In section \ref{sec5}, we show that the only
critical points of $\height$ on a certain locus $\SY_g\subset\SM_g$
arc at $\height^{-1}\{0\}\cap\SY$, proving the existence of the
desired surfaces. 
We define this locus $\SY_g\subset\SM_g$ as an extension
of a desingularization of the $(g-1)$-handled Scherk surface
$S_{g-1}$, viewed as an element of
$\SM_{g-1}\subset\p\ov{\SM_g}$, itself a stratum of the
boundary $\p\ov{\SM_g}$ of the closure $\ov{\SM_g}$ of
$\SM_g$. 

In section \ref{sec6}, we show that the resulting surfaces
$\{S_g\}$ are all embedded.

\section{Background and Notation}
\label{sec2}

\subsection{History of doubly-periodic Minimal Surfaces}
\label{sec21}

In 1835, Scherk \cite{sche1} discovered a $1$-parameter family  of 
properly embedded doubly-periodic
minimal surfaces $S_0(\theta)$ in euclidean space. These surfaces are
invariant 
under a lattice
$\Gamma=\Gamma_\theta$ of horizontal euclidean translations 
of the plane which induce
orientation-preserving isometries
of the surface $S_0(\theta)$. If we identify the $xy$-plane with $\BC$,
this lattice is spanned by 
vectors $1,
e^{i \theta}$.

In  the upper  half space, $S_0(\theta)$ is asymptotic to a family of equally spaced half
planes. The same holds in the lower half space for a different family of half planes. The angle between these
two families is the parameter $\theta\in(0,\pi/2]$. The quotient surface $S_0(\theta)/\Gamma_\theta$ is
conformally equivalent to a sphere punctured at $\pm1, \pm e^{i \theta}$.

Lazard-Holly and Meeks \cite{lhm} have shown that all embedded
genus 0 doubly-periodic surfaces belong to this family. 

Since then, many more properly embedded doubly-periodic
minimal surfaces  in euclidean space have been found:

Karcher \cite{ka4} and Meeks-Rosenberg \cite{mr4} constructed  a 3-dimensional family of genus-one examples where
the bottom and top planar ends are parallel. Some of these surfaces can be
visualized as a fence of Scherk towers.

P\'{e}rez, Rodriguez and Traizet  \cite{prt1} have shown that any doubly-periodic minimal surface of genus one with parallel ends belongs to this family.

The first attempts to add further handles to these surfaces failed, 
and similarly 
it seemed to be impossible to add
just one handle to Scherk's doubly-periodic surface between {\it every}
pair of planar ends.

However, Wei \cite{wei2} added another handle to Karcher's
examples (where all ends are parallel) by adding the handle between every
{\it second} pair of ends. This family has been generalized by
Rossman, 
Thayer and Wohlgemuth \cite{rtw1} to include more ends. 
Recently, Connor and Weber \cite{cowe1} adapted Traizet's regeneration
method to 
construct many examples of arbitrary genus and arbitrarily many ends.

Soon after Wei's example, Karcher found an orthogonally-ended  doubly-periodic
Scherk-type surface with handle by also adding the handle only between
every {\it second} pair of ends, see figure \ref{fig:scherk1(1)}.

Baginski and Ramos-Batista \cite{brb1} as well as Douglas \cite{dou1}
have shown that the Karcher example can be deformed to a 1-parameter
family by changing the angle between the ends.


On the theoretical side, Meeks and Rosenberg  \cite{mr3}
have shown the following:

\begin{theorem} A complete embedded  minimal surface in $\BE^3/\Gamma$
has only finitely many ends. In particular, it has finite topology if and only if 
it has finite genus.
\end{theorem}

\begin{theorem}
A complete embedded  minimal surface in $\BE^3/\Gamma$ has finite total curvature if and only 
if it has finite topology. In this case, the surface can be given by holomorphic \wei data 
on a compact Riemann surface with finitely many punctures which extend meromorphically to these punctures.
\end{theorem}

\subsection{\wei Representation} 

Let $S$ be a minimal surface in space with metric $ds$, and denote the underlying Riemann surface
by $\SR$. The stereographic projection of the \ga map defines a meromorphic function $G$ on
$\SR$, and the complex extension of the third coordinate differential $dx_3$ defines a holomorphic 1-form $dh$ on 
$\SR$, called the height differential.
The data $(\SR, G, dh)$ comprise the \wei data of the minimal surface. Via
$$
\aligned
\om_1&=\frac 12(G^{-1}-G) dh
\\
\om_2&=\frac i 2(G^{-1}+G)dh
\\
\om_3&=dh
\endaligned
$$
one can reconstruct the surface as
$$
z \mapsto \re \int_\cd^z
(\om_1,\om_2,\om_2)
$$
Vice versa, this {\em \wei representation} can be used on any set of \wei data
to define
a minimal surface in space.  
Care has to be taken that the metric becomes complete. 

This procedure works locally, but 
the surface is only well-defined globally if the periods

$$
\re \int_\g
\f12(G^{-1}-G) dh,\frac i 2(G^{-1}+G)
dh , dh)
$$
vanish for every cycle $\g \subset \SR$.  The problem of finding
compatible meromorphic data $(G, dh)$ which satisfies the above 
conditions on the periods of $\om_i$ is known as `the period 
problem for the \wei representation'.

These period conditions are equivalent to

\begin{equation}
\label{tag21a}
\re \int_\g dh = 0 
\end{equation}

and
\begin{equation}\label{tag21b}
\int_\g Gdh = \overline{ \int_\g G^{-1} dh}. 
\end{equation}

For surfaces that are intended to be periodic, one can either define 
\wei data on periodic surfaces, or more commonly, one can insist that
equations \eqref{tag21a} and \eqref{tag21b} hold for only some of the 
cycles, with the rest of the homology having periods that generate
some
discrete subgroup of Euclidean translations. Our setting will be of
the latter type, with periods that either vanish or are in a rank-two 
abelian group of orthogonal horizontal translations.

\subsection{Flat Structures}

The forms $\om_i$ lead to singular flat structures on the
underlying Riemann surfaces, defined via the line elements
$ds_{\om_i} = \vert \om_i\vert$.  These singular metrics are
flat away from the support of the divisor of $\om_i$; on
elements $p$ of that divisor, the metrics  have cone points
with angles equal to $2\pi(ord_{\om_i}(p) + 1)$.  More
importantly, the periods of the forms are given by the
Euclidean geometry of the  developed image of the metric
$ds_{\om_i}$ --- a period of a cycle $\g$ is the (complex) distance 
$\BC$ between consecutive images of a distinguished point in
$\g$.  We reverse this procedure in Section \ref{sec3}: we use
putative developed images of the one-forms $Gdh$, $G^{-1}
dh$, and $dh$ to  solve formally the period problem for some
formal \wei data. For more details about the properties of flat structures 
associated to meromorphic 1-forms in connection with minimal surfaces, see
\cite{whw1}.

\subsection{The Conjugate Plateau Construction and Krust's Theorem}
\label{sec211} 

The material here will be needed in Section \ref{sec6} where we will prove the embeddedness of
our surfaces. General references for the cited theorems of this subsection are \cite{os1} and \cite{dhkw1}.

Given a minimal immersion
$$
F:z \mapsto \re \int^{z} \omega \quad,
$$
then the immersions
$$
F_t:z \mapsto \re \int^{z} e^{i t} \omega
$$
define the {\it associate family} of  minimal surfaces.  Among them, the {\it conjugate}
surface
$F^*=F_{\pi/2}$ is of special importance because symmetry properties of $F$ correspond to
symmetry properties of $F^*$ as follows:

\begin{theorem}
If a minimal surface patch  is bounded by a straight line, the conjugate patch is
bounded by a planar symmetry curve, and vice versa. Angles at corresponding vertices are the
same.

If $\ell_1$ and $\ell_2$ are a pair of intersecting straight lines on the conjugate patch
corresponding to the intersection of a pair of (planar) symmetry
curves lying on planes $P_1$ and $P_2$, then the lines $\ell_1$ and
$\ell_2$ span a plane orthogonal to the line common to $P_1$ and $P_2$.
\end{theorem}

\begin{proof}  The first paragraph is well-known: see \cite{ka6} for
example. The second paragraph is elementary, for if $P$ is a plane
of reflective symmetry, then the normal to the surface must lie in
the plane. At the intersection of two such planes, the normal must
lie in both planes, hence in the line $L$ of intersection of the two
planes.  But the
Gauss map is preserved by the conjugacy correspondence, hence both 
of the corresponding straight lines $\ell_1$ and $\ell_2$
are orthogonal to $L$.  Thus the plane spanned by $\ell_1$ and
$\ell_2$ is normal to $L$, the line of intersection of $P_1$ and
$P_2$.
\end{proof}

The best-known example of a conjugate pair are the catenoid and one full turn of the helicoid.

The second-best-known examples are the singly- and doubly-periodic Scherk surfaces.

To get started with the conjugate Plateau construction, one can take a boundary contour
bounded by straight lines and solves the Plateau problem using the classic result of Douglas and Rad\'o (see \cite{la1} for a proof):

\begin{theorem} 
Let $\Gamma$ be a Jordan curve in $\BE^3$  bounding a
finite-area disk. Then there exists a continuous map $\psi$ from the closed unit disk $\bar D$
into  
$\BE^3$ such that
\begin{enumerate}
\item $\psi$ maps $S^1=\partial D$ monotonically onto $\Gamma$.
\item $\psi$ is harmonic and almost conformal in $D$.
\item $\psi(\bar D)$ minimizes the area among all admissible maps.
\end{enumerate}
\end{theorem}

Here  {\it almost conformal} allows a vanishing derivative and {\it admissible maps} on the disk
are required to be in $H^{1,2}(D, \BE^3)$ so that their trace on $\partial D$ can be represented
by a weakly monotonic, continuous mapping $\partial D \to \Gamma$.

For {\it good} boundary curves, one obtains the embeddedness and uniqueness of the Plateau
solution for free by

\begin{theorem}
If $\Gamma$ has a one-to-one parallel projection onto a planar convex
curve, then $\Gamma$ bounds at most one disk-type minimal surface which can be expressed as
the graph of a function $f: \BE^2\to\BE^3$. 
\end{theorem}

The embeddedness of a Plateau solution sometimes implies
the embeddedness of the conjugate surface. This observation is due to Krust (unpublished),
see
\cite{ka6}.

\begin{theorem}[Krust]
If a minimal surface is a graph over a convex domain, then the
conjugate piece is also a graph.
\end{theorem}

\subsection{\tec Theory}
\label{sec22}

For $M$ a smooth surface, let Teich$(M)$ denote the \tec
space of all conformal structures on $M$ under the
equivalence relation given by pullback by diffeomorphisms
isotopic to the identity map id:~$M\lra M$. Then it is
well-known that Teich$(M)$ is a smooth finite-dimensional
manifold if $M$ is a closed surface.

There are two spaces of tensors on a Riemann surface $\SR$
that are important for the \tec theory. The first is the
space QD$(\SR)$ of holomorphic quadratic differentials, i.e.,
tensors which have the local form $\Phi=\vp(z)dz^2$ where
$\vp(z)$ is holomorphic. The second is the space of Beltrami
differentials Belt$(\SR)$, i.e., tensors which have the local
form $\mu=\mu(z)d\bar z/dz$.

The cotangent space $T^*_{[\SR]}$(Teich$(M)$) is canonically
isomorphic to QD$(\SR)$, and the tangent space is given by
equivalence classes of (infinitesimal) Beltrami differentials, where
$\mu_1$ is equivalent to $\mu_2$ if
$$
\int_\SR\Phi(\mu_1-\mu_2) = 0\qquad\text{for every }\
\Phi\in\text{QD}(\SR).
$$

If $f:\BC \to \BC$ is a diffeomorphism, then the Beltrami
differential associated to the pullback conformal structure
is $\nu = \frac{f_{\bar z}}{f_z} \frac{d\bar z}{dz}$. If $f_{\e}$ is a
family of such diffeomorphisms with $f_0= id$, then the
infinitesimal Beltrami differential is given by

$$
\frac d{d\e}\bigm|_{\e=0}\nu_{f_\e}=\left(\frac {d}{d\e}\bigm|_{\e=0} f_\e\right)_{\bar z}
$$

We will
carry out an example of this computation in section \ref{sec52}

A holomorphic quadratic differential comes with a picture that is
a useful aid to one's intuition about it. The picture is that of
a pair of transverse measured foliations, whose properties we sketch
briefly (see \cite{flp} for more details).

A $C^k$ measured foliation on $\SR$ with singularities
$z_1,\dots,z_l$ of order $k_1,\dots,k_l$ (respectively) is given by
an open covering $\{U_i\}$ of $\SR-\{z_1,\dots,z_l\}$ and open sets
$V_1,\dots,V_l$ around $z_1,\dots,z_l$ (respectively) along with
real valued  $C^k$ functions $v_i$ defined on $U_i$ s.t.
\begin{enumerate}
\item $|dv_i|=|dv_j|$ on $U_i\cap U_j$
\item $|dv_i|=|\im(z-z_j)^{k_j/z}dz|$ on $U_i\cap V_j$
\end{enumerate}

Evidently, the kernels $\ker dv_i$ define a $C^{k-1}$ line field on
$\SR$ which integrates to give a foliation $\SF$ on
$\SR-\{z_1,\dots,z_l\}$, with a $k_j+2$ pronged singularity at $z_j$.
Moreover, given an arc $A\subset\SR$, we have a well-defined measure
$\mu(A)$ given by
$$
\mu(A) = \int_A |dv|
$$
where $|dv|$ is defined by $|dv|_{U_i}=|dv_i|$. An important feature
that we require
of this measure is its ``translation invariance''. That is, suppose
$A_0\subset\SR$ is an arc transverse to the foliation $\SF$,
with $\p A_0$ a pair of points, one on the leaf $l$ and one on the
leaf $l'$; then, if
we deform $A_0$ to $A_1$ via an isotopy through arcs $A_t$
that maintains the
transversality of the image of $A_0$ at every time, and also
keeps the endpoints of the arcs $A_t$ fixed on the leaves
$l$ and $l'$, respectively, then we require that
$\mu(A_0)=\mu(A_1)$.

Now a holomorphic quadratic differential $\Phi$ defines a measured
foliation in the following way. The zeros $\Phi^{-1}(0)$ of $\Phi$
are well-defined; away from these zeros, we can choose a canonical
conformal coordinate $\z(z)=\int^z\sqrt\Phi$ so that $\Phi=d\z^2$.
The local measured foliations ($\{\re\z=\oper{const}\}$,
$|d\re\z|$) then piece together to form a measured foliation known
as the vertical measured foliation of $\Phi$, with the
translation invariance of this measured foliation of $\Phi$
following from Cauchy's theorem.

Work of Hubbard and Masur \cite{hum1} (see also alternate proofs
in \cite{ker1,gar1, wo98}, following Jenkins \cite{j1}
and Strebel \cite{stre1}, showed that given a measured foliation
$(\SF,\mu)$ and a Riemann surface $\SR$, there is a unique holomorphic
quadratic differential $\Phi_\mu$ on $\SR$ so that the horizontal
measured foliation of $\Phi_\mu$ is equivalent to  $(\SF,\mu)$.

\subsection{Extremal length}
The extremal length $\ext_\SR([\g])$ of a class of
arcs $\G$ on a Riemann surface $\SR$ is defined to be the conformal
invariant
$$
\sup_\rho\f{\ell^2_\rho(\G)}{\text{Area}(\rho)}
$$
where $\rho$ ranges over all conformal metrics on $\SR$ with areas
$0<\text{Area}(\rho)<\infty$ and $\ell_\rho(\G)$ denotes the infimum
of $\rho$-lengths of curves $\g\in\G$. Here $\G$ may consist of all
curves freely homotopic to a given curve, a union of free homotopy
classes, a family of arcs with endpoints in a pair of given
boundaries, or even a more general class.  Kerckhoff \cite{ker1}
showed that this definition of extremal lengths of curves extended
naturally to a definition of extremal lengths of measured foliations.

For a class $\G$ consisting of all curves freely homotopic to a
single curve $\g\subset M$, (or more generally, a measured foliation
$(\SF, \mu)$) we see that $\ext_{(\cd)}(\G)$ (or $\ext_{(\cd)}(\mu)$)
can be construed as a real-valued function $\ext_{(\cd)}(\G)$:
Teich$(M)\lra\BR$. Gardiner \cite{gar1} showed that
$\ext_{(\cd)}(\mu)$ is differentiable and Gardiner and Masur \cite{gama1}
showed
that $\ext_{(\cd)}(\mu)\in C^1$ (Teich$(M)$).
In our particular applications, the extremal length
functions on our moduli spaces will be real analytic: this
will be explained in Proposition \ref{prop:extanalytic}.

Moreover Gardiner
computed that
$$
d\ext_{(\cd)}(\mu)\bigm|_{[\SR]} = 2\Phi_{\mu}
$$
so that
\begin{equation}\label{tag22}
\left(d\ext_{(\cd)}(\mu)\bigm|_{[\SR]}\right)[\nu] =
4\re\int_\SR\Phi_\mu\nu.
\end{equation}

\subsection{A Brief Sketch of the Proof}
\label{sec23}

In this subsection,
we sketch basic logic of the approach and the ideas of the proofs, as a step-by-step recipe.

{\bf Step 1. Draw the Surface.} The first step in proving the 
existence of a minimal surface is to work out a detailed proposal.
This can either be done numerically, as in the work of
(i) Thayer \cite{tha1} for the Chen-Gackstatter surfaces we
discussed in \cite{ww1}, (ii) Boix
and Wohlgemuth \cite{bow1,w2,w3,w4} for the low genus surfaces we treated
in \cite{ww2} and (iii) Figures \ref{fig:scherk1} and \ref{fig:scherk1(1)} below for the
present case;
or it can be schematic, showing how various portions of the surface
might fit together, using plausible symmetry assumptions.

{\bf Step 2. Compute the Divisors for the Forms $Gdh$
and $G^{-1}dh$.} From the model that we drew in Step 1, we can
compute the divisors for the \wei data, which we just defined
to be the \ga map $G$ and the 'height' form $dh$. (Note here
how important it is that the \wei representation is given in 
terms of geometrically defined quantities --- for us, this gives the
passage between the extrinsic geometry of the minimal surface 
as defined in Step 1 and the conformal geometry and \tec
theory of the later steps.) Thus we can also compute the
divisors for the meromorphic forms $Gdh$ and $G^{-1}dh$ on the
Riemann surface (so far undetermined, but assumed to exist)
underlying the minimal surface. Of course the divisors
for a form determine the form up to a constant, so the divisor
information nearly determines the \wei data for our surface.
Here our schematics suggest the appropriate divisor information,
and this is confirmed by the numerics.

{\bf Step 3. Compute the Flat Structures for the Forms $Gdh$
and $G^{-1}dh$ required by the period conditions.} 
A meromorphic form on a Riemann surface defines
a flat singular (conformal) metric on that surface: for example,
from the form $Gdh$ on our putative Riemann surface, we determine
a line element $ds_{Gdh}= |Gdh|$. This metric is locally Euclidean
away from the support of the divisor of the form and has a 
complete  Euclidean cone structure in a neighborhood of a zero or pole of the
form.  Thus we can develop the universal cover of the surface
into the Euclidean plane.

The flat structures for the forms $Gdh$ and $G^{-1}dh$ are
not completely arbitrary: because the periods for the pair of
forms must be conjugate (formula \ref{tag21b}), the flat structures must
develop into domains which have a particular
Euclidean geometric relationship to one another. This relationship is
crucial to our approach, so we will dwell on it somewhat. If the
map $\dev:\Omega \lra \BE^2$ is the map which develops the 
flat structure of a form, say $\alpha$, on a domain 
$\Omega$ into $\BE^2$, then the map $\dev$ pulls back the canonical
form $dz$ on $\BC \cong \BE^2$ to the form $\alpha$ on $\omega$.
Thus the periods of $\alpha$ on the Riemann surface are given
by integrals of $dz$ along the developed image of paths in $\BC$,
i.e. by differences of the complex numbers representing 
endpoints of those paths in $\BC$.

We construe all of this as requiring that the flat structures 
develop into domains that are ``conjugate'': if we collect all
of the differences in positions of parallel sides for the developed image
of the form $Gdh$ into a large complex-valued 
$n$-tuple $V_{Gdh}$, and we collect
all of the differences in positions of 
corresponding parallel sides for the developed image
of the form $G^{-1}dh$ into a large 
complex-valued n-tuple $V_{G^{-1}dh}$, then these
two complex-valued vectors  $V_{Gdh}$ and $V_{G^{-1}dh}$ should
be conjugate. Thus, we translate the ``period problem'' into
a statement about the Euclidean geometry of the developed
flat structures. This is done at the end of section \ref{sec3}.

The period problem \ref{tag21a} for the form $dh$ will be trivially
solved for the surfaces we treat here.

{\bf Step 4. Define the moduli space of pairs of conjugate flat
domains.}  Now we work backwards.  We know the general form of 
the developed images (called $\ogup$ and $\ogdn$, respectively)
of flat structures associated to the forms
$Gdh$ and $G^{-1}dh$, but in general, there are quite a few 
parameters of the flat structures left undetermined; this holds 
even after
we have assumed symmetries, determined the \wei divisor data 
for the models and used the period conditions \ref{tag21b} to 
restrict the relative Euclidean geometries of the pair
$\ogup$ and $\ogdn$. Thus, there is a moduli space $\Delta$ of 
possible candidates of pairs $\ogup$ and $\ogdn$: our 
period problem (condition \ref{tag21b}) is now a conformal problem
of finding such a pair which are conformally equivalent
by a map which preserves the corresponding cone points.
(Solving this problem means that there is a well-defined
Riemann surface which can be developed into $\BE^2$ in
two ways, so that the pair of pullbacks of the form $dz$ give
forms $Gdh$ and $G^{-1}dh$ with conjugate periods.)

The condition of conjugacy of the domains $\ogup$ and $\ogdn$
often dictates some restrictions on the moduli space, and even
a collection of geometrically defined coordinates.  We work these
out in section \ref{sec3}. 

{\bf Step 5. Solve the Conformal Problem using \tec theory.}
At this juncture, our minimal surface problem has become a 
problem in finding a special point in a product of moduli
spaces of complex domains: we will have no further references
to minimal surface theory. The plan is straightforward: we will
define a height function $\height: \Delta \lra \BR$
with the properties:

\begin{enumerate}
\item  (Reflexivity) The height $\height$ equals $0$ only at a solution
to the conformal problem 
\item  (Properness) The height $\height$ is proper on $\Delta$.
This ensures the existence of a critical point.
\item  (Non-degenerate Flow) If the height $\height$ at a pair $(\ogup, \ogdn)$
does not vanish, then the height $\height$ is not critical at 
that pair $(\ogup, \ogdn)$.
\end{enumerate}

This is clearly enough to solve the problem: we now sketch the
proofs of these steps.

{\bf Step 5a. Reflexivity.} We need conformal invariants of a 
domain that provide a complete set of invariants for Reflexivity,
have estimable asymptotics for Properness, and computable 
first derivatives (in moduli space) for the Non-degenerate Flow property.
One obvious choice is a set of functions of 
extremal lengths for a good choice
of curve systems, say $\Gamma = \{\gamma_1, \dots, \gamma_g\}$ 
on the domains.
These are defined for our examples in section \ref{sec41}.
We then define a height function $\height$ which vanishes only when there
is agreement between all of the extremal lengths
$\ext_{\ogup}(\gamma_i)= \ext_{\ogdn}(\gamma_i)$ and which 
blows up when $\ext_{\ogup}(\gamma_i)$ and $\ext_{\ogdn}(\gamma_i)$
either decay or blow up at different rates. See for example
Definition \ref{def412} and Lemma \ref{lemma435}.

{\bf Step 5b Properness.} Our height function will measure 
differences in the extremal lengths $\ext_{\ogup}(\gamma_i)$ and
$\ext_{\ogdn}(\gamma_i)$. A geometric
degeneration of the flat structure of either $\ogup$ or
$\ogdn$ will force one of the extremal lengths $\ext_{\bullet}(\gamma_i)$
to tend to zero or infinity, while the other 
extremal length stays finite and
bounded away from zero.  This is a straightforward situation where
it will be obvious that the height function will blow up.
A more subtle case arises when a geometric degeneration of the
flat structure forces {\it both} of the extremal lengths
$\ext_{\ogup}(\gamma_i)$ and $\ext_{\ogdn}(\gamma_i)$ to
{\it simultaneously} decay (or explode). In that case, we begin
by observing that there is a natural map between the
vector $<\ext_{\ogup}(\gamma_i)>$ and the vector
$<\ext_{\ogdn}(\gamma_i)>$. This pair of vectors is 
reminiscent of pairs of solutions to a hypergeometric
equation, and we show, by a monodromy argument analogous to
that used in the study of those equations, that it is not
possible for corresponding components of that vector to
vanish or blow up at identical rates. In particular, we show
that the logarithmic terms in the asymptotic expansion of the
extremal lengths near zero have a different sign, and this
sign difference forces a difference in the rates of decay
that is detected by the height function, forcing it to blow
up in this case.  The monodromy argument is given in section \ref{sec43},
and the properness discussion consumes section \ref{sec42}.

{\bf Step 5c. Non-degenerate Flow.} The domains $\ogup$ and $\ogdn$
have a remarkable property: if
$\ext_{\ogup}(\gamma_i) > \ext_{\ogdn}(\gamma_i)$, then when 
we deform $\ogup$ so as to decrease $\ext_{\ogup}(\gamma_i)$,
the conjugacy condition forces us to deform $\ogdn$ so as
to increase $\ext_{\ogdn}(\gamma_i)$.  We can thus always deform
$\ogup$ and $\ogdn$ so as to reduce one term of the height
function $\height$. We develop this step in Section \ref{sec5}.

{\bf Step 5d. Regeneration.} In the process described in 
the previous step, an issue arises: we might be able to 
reduce one term of the height function via a deformation, but
this might affect the other terms, so as to not provide
an overall decrease in height.  We thus seek a locus $\SY$ in
our moduli space where the height function has but a single
non-vanishing term, and all the other terms vanish to at least
second order.  If we can find such a locus $\SY$, we can flow along
that locus to a solution.  To begin our search for such a locus,
we observe which flat domains arise as limits of our
domains $\ogup$ and $\ogdn$: commonly, the degenerate domains 
are the flat domains for a similar minimal surface problem,
maybe of slightly lower genus or fewer ends.

We find our desired locus
by considering the boundary of the (closure) of the moduli
space $\Delta$: this boundary has strata of moduli spaces $\Delta'$
for minimal surface problems of lower complexity.  By induction,
there are solutions of those problems represented on such
a boundary strata $\Delta'$ (with all of the
corresponding extremal lengths in agreement), 
and we prove that there is a nearby
locus $\SY$ inside the larger moduli space $\Delta$ which 
has the analogues of those same extremal lengths in agreement.
As a corollary of that condition, the height function on $\SY$
has the desired simple properties.

\subsection {The Geometry of Orthodisks}
\label{sec24}

In this section we introduce the notion of orthodisks.

Consider the upper half plane $\BH$ and $n\ge3$ distinguished points
$t_i$ on the real line. The point $t_\infty=\infty$ will also
be a distinguished point. We will refer to the upper half plane together
with these data as a {\it conformal polygon}
and to the distinguished points as
{\it vertices}. Two conformal polygons are {\it conformally equivalent} if
there is
a biholomorphic map between the disks carrying vertices to vertices,
and fixing $\infty$.

Let
$a_i$ be some odd integers such that

\begin{equation}
\label{tag31}
a_\infty=-4-\sum_i a_i 
\end{equation}

By a Schwarz-Christoffel
map we mean the map
\begin{equation}
\label{tag32}
F:z \mapsto \int_i^z
(t-t_1)^{a_1/2}\cdot\ldots\cdot(t-t_n)^{a_n/2} dt 
\end{equation}

A point $t_i$ with $a_i>-2$ is called {\it finite}, otherwise {\it infinite}.
By Equation \ref{tag31}, there is at least one finite vertex.

\begin{definition}
\label{def241} Let $a_i$ be odd integers.
The pull-back of the flat metric on $\BC$ by $F$ defines a
complete flat metric with boundary on $\BH \cup \BR$  without the
infinite vertices. We call such a metric an {\it orthodisk}. The $a_i$ are
called the {\it vertex data} of the orthodisk. The {\it edges} of an orthodisk
are the boundary segments between vertices; they come in a natural order.
Consecutive edges meet orthogonally at the finite vertices.
Every other edge is parallel
under the parallelism induced by the flat metric of the
orthodisk. Oriented distances between parallel edges are called {\it periods}.
We will discuss the relationship of these periods to the periods arising in the minimal surface context
in Section \ref{sec3}.

The periods can have
$4$ different signs:
$+1,-1,+i,-i$.
\end{definition}

The interplay between these signs is crucial to our monodromy
argument, especially Lemma \ref{lemma435}.

\begin{remark} The integer $a_i$ corresponds to an angle $(a_i+2)\pi/2$
of the orthodisk. Negative angles are meaningful because 
a vertex (with a negative angle $-\theta$) lies at infinity
and is the intersection of a pair of lines which also
intersect at a finite point, where they make a {\it positive} angle of
$+\theta$.
\end{remark}

In all the drawings of the orthodisks to follow, we mean the domain to be
to the left
of the boundary, where we orient the boundary by the order of the
points $t_i$.

\subsection {Scherk's and Karcher's Doubly-Periodic Surfaces}
\label{sec25}

The singly-  and doubly-periodic Scherk surfaces are conjugate spherical
minimal surfaces whose \wei data lead to no computational
difficulties and whose orthodisk description illustrates
the basic concepts in an ideal way.

We discuss first the \wei representation of  
the doubly-periodic Scherk surfaces $S_0(\theta)$
(see figure \ref{fig:scherk1}).

$$
G(z) = z
$$
and
$$
dh = z^{-1} \frac{i dz}{z^2+z^{-2}-2\cos{2\phi}}
$$ 
on the Riemann sphere punctured at the four points $q=\pm
e^{\pm i\phi}$.

\begin{figure}[h] 
\centering
\includegraphics[width=2in]{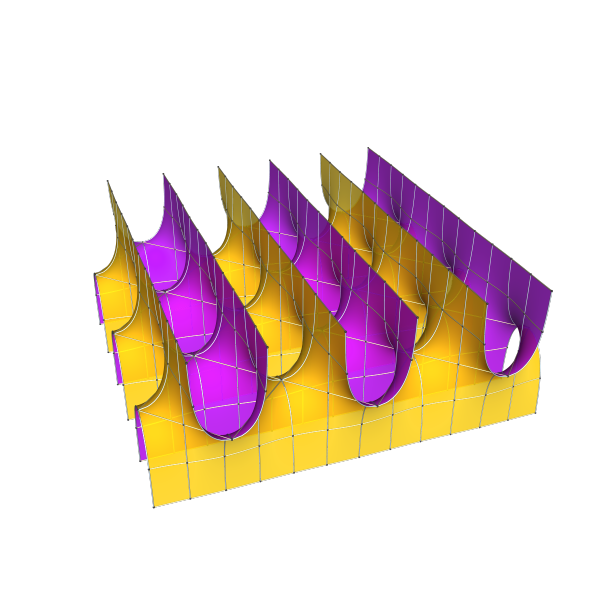} 
\caption{Scherk's surface}
\label{fig:scherk1}
\end{figure}

The residues of $dh$ at the punctures have the real values
$\frac{\pm 1}{4 \sin{2\phi}}$ and hence we have no vertical
periods. At the punctures (which correspond to the ends) the
\ga map is horizontal and takes the values $\pm e^{\pm
i\phi}$ so that the angle between two ends is $2\phi$. 
Because of this, the horizontal surface periods around a
puncture $q$ are given as complex numbers by
$$
\aligned
\re &\oint_q \frac i2\left(G + \frac1G\right) dh +i
\re\oint_q\frac 12\left(\frac1G-G\right) dh =\\
&= \re\left(\pi i^2\left(G(q)+\overline{G(q)}\right) \res_q
dh\right) - i \re\left(\pi
i \left(G(q)-\overline{G(q)}\right) \res_q dh\right) \\
&= -2\pi \overline{G(q)} \res_q dh.
\endaligned
$$
These numbers span a horizontal lattice in $\BC$ so that the surface 
is indeed doubly-periodic. Indeed we can now regard the
result as being defined over the even squares of a sheared
checkerboard with vertices given by the period lattice.

Our principal interest in this paper will be with 
handle addition for the {\it orthogonal} Scherk surface $S_0=S_(\pi/2)$, i.e
the case where in the above we set $\phi = \pi/2$ to obtain
a period lattice which is a multiple of the Gaussian 
integers.

It would be interesting to {\it shear} these surfaces, as in the 
work of Ramos-Batista-Baginsky \cite{brb1} or Douglas \cite{dou1},
so that the periods span a non-orthogonal horizontal lattice,
or to add handles to a sheared surface $S_1(\theta)$.

\begin{figure}[h] 
\centering
\includegraphics[width=2in]{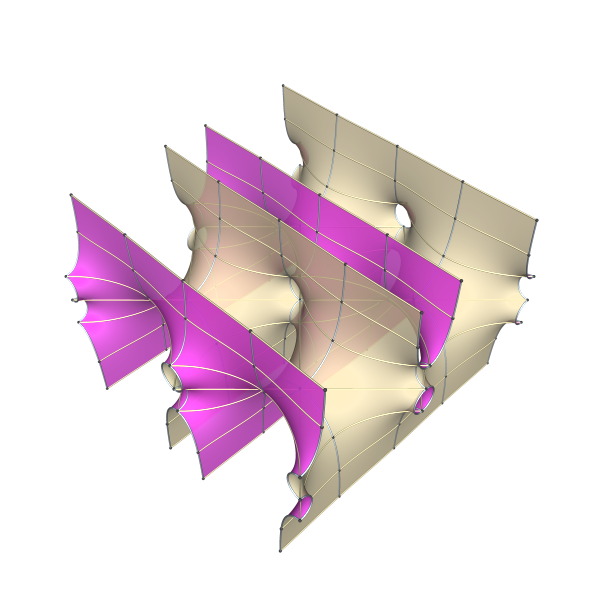} 
\caption{Scherk's surface with an additional handle}
\label{fig:scherk1(1)}
\end{figure}

The construction of $S_1$ is due to
Hermann Karcher \cite{ka4} who found a way to `add a handle' to
the classical Scherk surface. Here we mean that he found a doubly-periodic
minimal surface  whose fundamental domain is equivariantly
isotopic to a doubly-periodic surface formed by adding a
handle to the Scherk surface above.

We now reprove the
result of Karcher from the perspective of orthodisks.

The quotient surface of $S_1$  by its horizontal period lattice is a square
torus with four punctures corresponding to the ends. We will construct this surface
using the  \wei data given by figure \ref{fig:s1divs} below. We begin
with the figure on the far right. The points with
labels $1$ to $4$ are 2-division points on the torus and
correspond to the points with vertical normal on $S_1$. The points
$E_1, E_2$ and their symmetric counterparts (not labelled)
correspond to the ends. The point $E_1$ is placed on the straight
segment between $1$ and $2$, and its position is a free
parameter that will be used to solve the period problem. The
other poles are then determined by the reflective symmetries
of the square.

\begin{figure}[h] 
\centering
\includegraphics[width=4.5in]{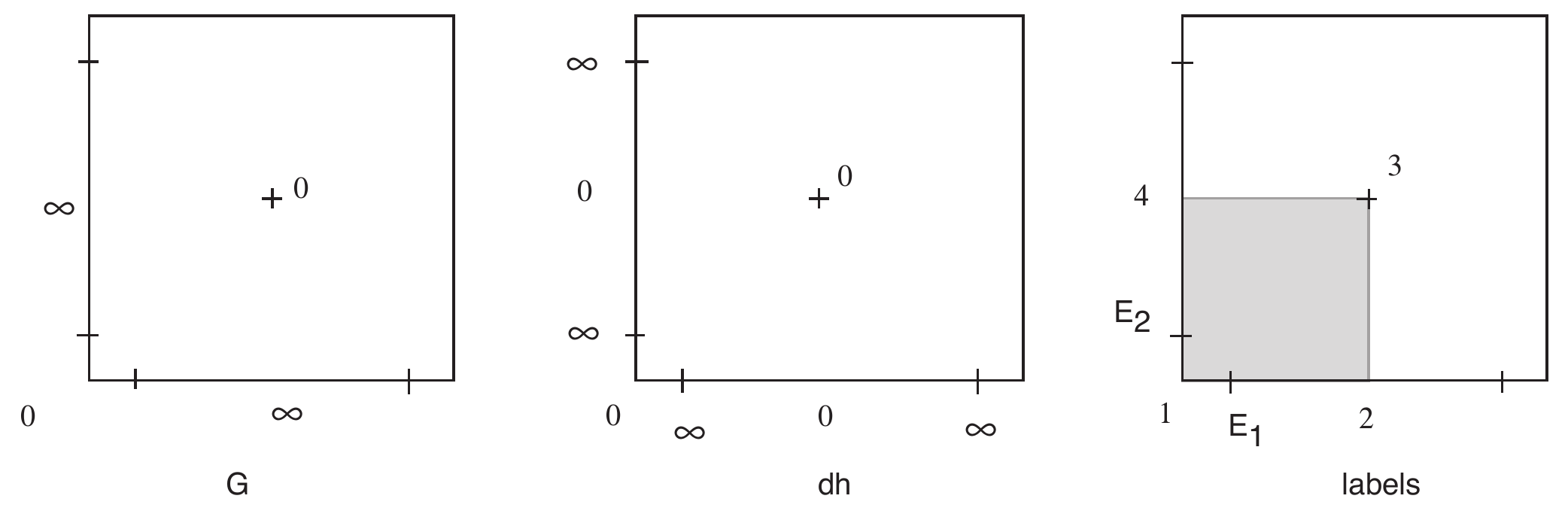} 
\caption{Divisors for the doubly-periodic Scherk surface with
handle}
\label{fig:s1divs}
\end{figure}

We obtain the  domains $\vert G dh\vert$ and $\vert\frac1G dh \vert$ by
developing {just} the
shaded square the regions given by figure \ref{fig:s1divs}.

\begin{figure}[h] 
\centering
\includegraphics[width=3.5in]{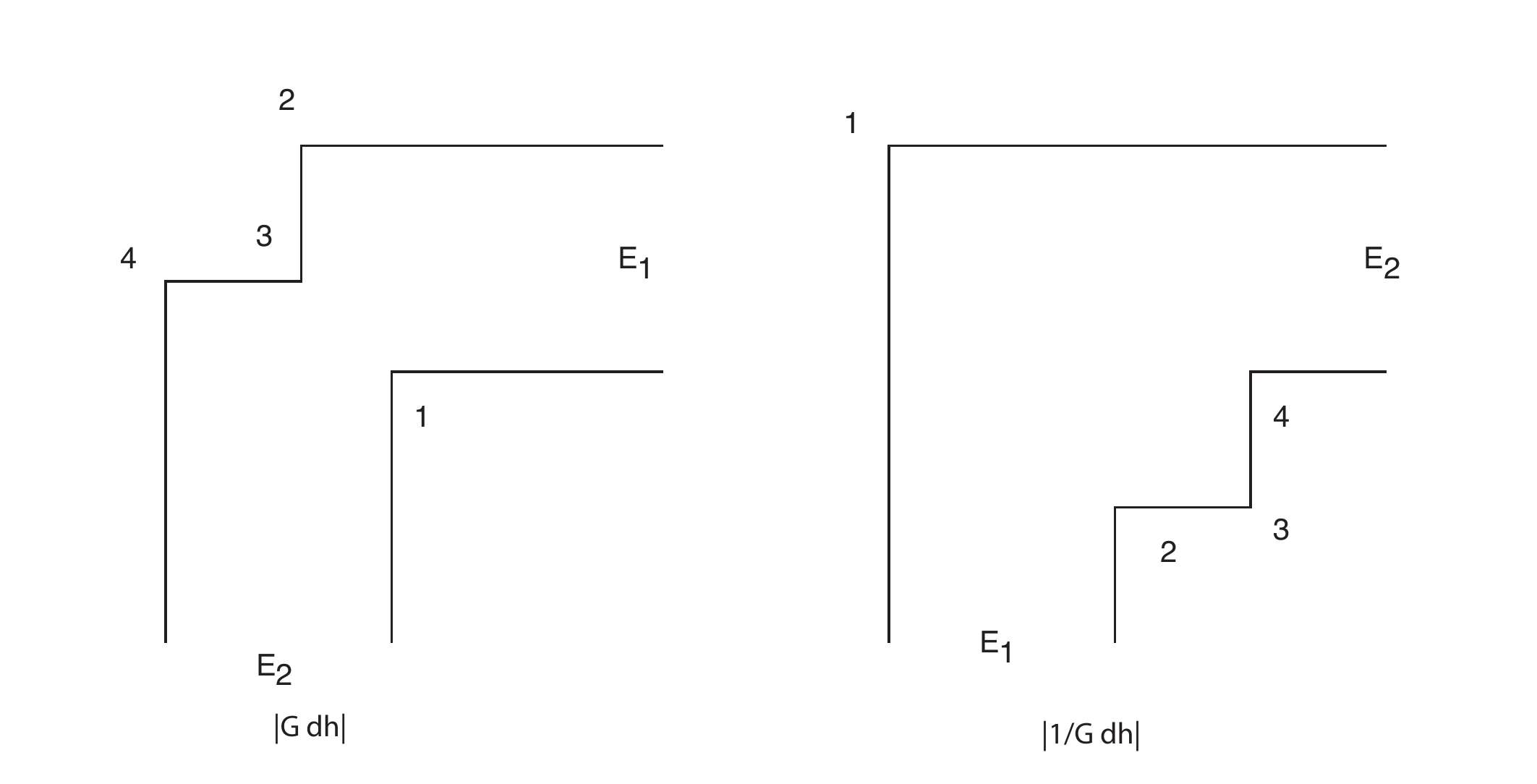} 
\caption{$\vert G dh \vert$ and $\vert \frac1G dh \vert$ orthodisks of a quarter piece}
\label{fig:s1gdh}
\end{figure}

As usual, the domains lie in separate planes and the
half-strips extend to infinity.

Observe that these domains are arranged to be symmetric with
respect to the $y=-x$ diagonal.

\begin{remark}  Four copies of the (say) $\vert G dh \vert$
domain fit together to form a  region as in Figure \ref{fig:fulls1} with
orthogonal half-strips, where a square from the center has
removed. This square (with opposite edges identified)
corresponds precisely to the added handle.

\begin{figure}[h] 
\centering
\includegraphics[width=3.5in]{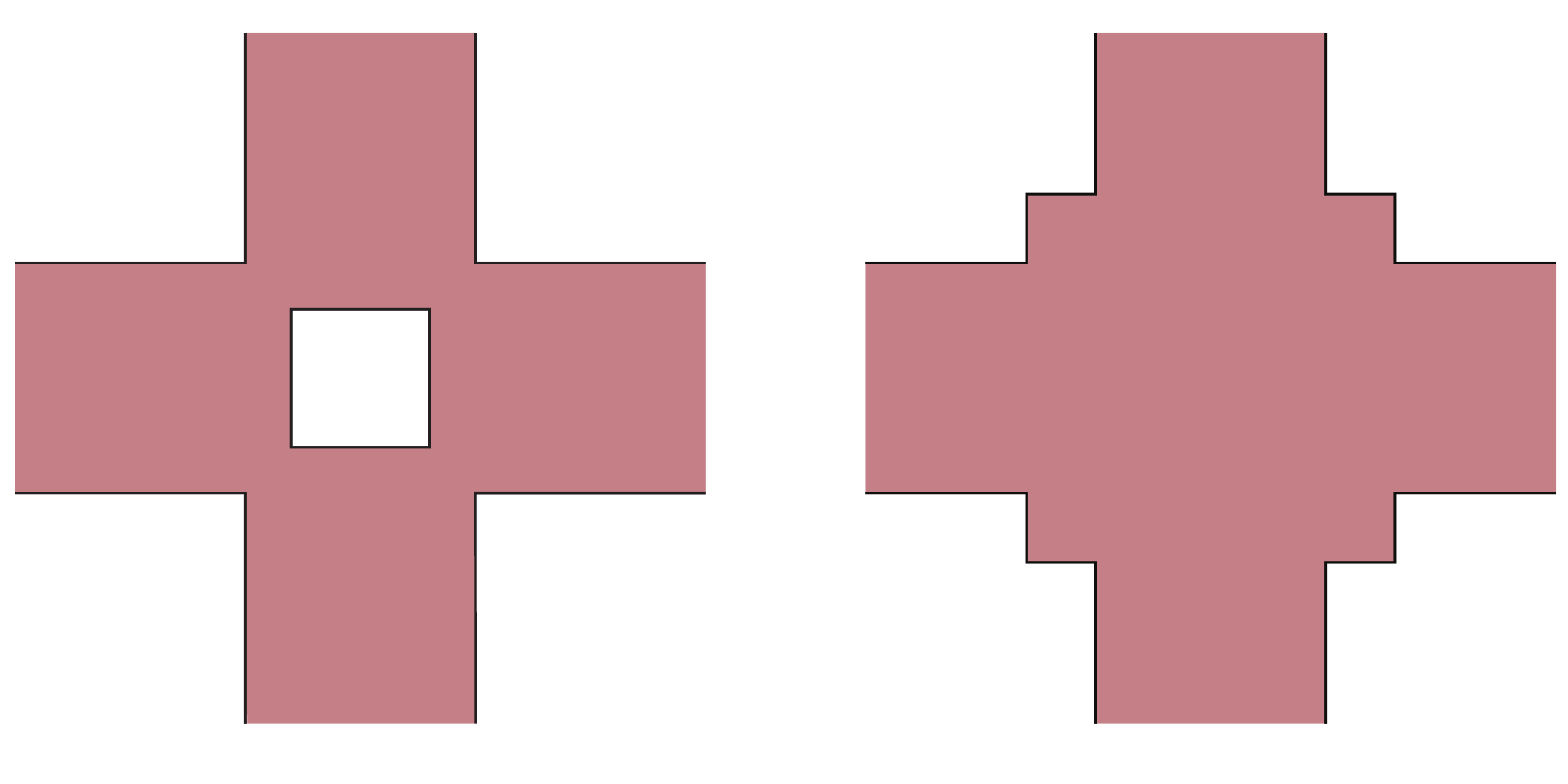} 
\caption{$\vert G dh \vert$ and $\vert \frac1G dh \vert$ orthodisks of a fundamental domain}
\label{fig:fulls1}
\end{figure}

\end{remark}

The period condition requires $G dh$ and $\frac1G dh$ to have
the same residues at $E_1$ and $E_2$ (so that the half-infinite strips need to
have the same width) and the remaining periods need to be
complex conjugate so that the $23$ and $34$ edges need to have
the same length in both domains. 

This is a 
one-dimensional period problem, and as is often the case, one
can solve it via an intermediate value argument. There are 
two versions of this argument, one approaching the problem
from the perspective of the period integrals on a fixed 
Riemann surface, and the other from the perspective of the
conformal moduli of the pair of orthodisks. We begin with the
version  based on the behavior of the period integrals in the
limit cases. We keep the discussion as close as possible to
the orthodisk description by  using Schwarz-Christoffel
maps from the upper half plane to parametrize the  domains
for $G dh$ and $\frac1G$:

\begin{align*}
f_{G dh}: z &\mapsto \int^z  {\frac{{\sqrt{r}}\,{\sqrt{-1 + {x^2}}}}
{{\sqrt{-1 + {r^2}}}\,{\sqrt{x}}\,
\left( -{r^2} + {x^2} \right) }} dx 
\\
f_{1/G dh}: z &\mapsto \int^z {\frac{{\sqrt{-1 + {r^2}}}\,{\sqrt{x}}}
{{\sqrt{r}}\,{\sqrt{-1 + {x^2}}}\,
\left( -{r^2} + {x^2} \right) }} dx
\end{align*}

Here the point $-1<-r<0<r<1<\infty$ on the real axis are
mapped to the labels $4,E_2,1, E_1, 2, 3$, respectively. The
normalization is choosen so that

$$
\res_{r} G dh = \frac1{2r}= \res_r \frac1G dh
$$

The parameter $r$ determines the relative position of $E_1$ and will now be
determined.
For this we compute the lengths of the edge $23$ as functions of $r$ using
hypergeometric functions as follows:

\begin{alignat*}{5}
A_{G dh} &=& \int_1^\infty  \frac1G dh &=&\sqrt\pi\frac{\sqrt r}{\sqrt{1-r^2}}
\frac{\Gamma(5/4)}{\Gamma(7/4)}F\left(\frac14,1,\frac74,r^2\right)
\\
A_{1/G dh} &=&\int_1^\infty  G dh &=& 2\sqrt\pi\frac{\sqrt{1-r^2}}{\sqrt r}
\frac{\Gamma(3/4)}{\Gamma(1/4)}F\left(\frac34,1,\frac54,r^2\right)
\end{alignat*}

The period condition requires
$$
A_{G dh}(r)=A_{1/G dh}(r) \ .
$$

To determine $r$, notice the `boundary conditions' 

\begin{align*}
A_{G dh}(0) &= \infty
\\
A_{1/G dh}(0) &= 0
\\
A_{G dh}(1) &= \frac\pi2
\\
A_{1/G dh}(1) &= \infty
\end{align*}
so that the intermediate value theorem implies the existence
of a solution.

Alternatively, we can give an 
intermediate value theorem argument based on
extremal length.  This is more in keeping with 
our theme of converting the period problem for 
minimal surfaces into a conformal problem for
orthodisks.

In terms of the orthodisks, the family of possible 
pairs $\{\ogup, \ogdn\}$ of orthodisks can be normalized
so that each half-strip end of either $\ogup$ or $\ogdn$
has width one. Then the family of pairs 
$\{\ogup, \ogdn\}$ is parametrized by the distance, say $d_{13}$,
between the points $1$ and $3$ in the domain $\ogup$:
there is degeneration in the domain $\ogup$ as
$d_{13} \lra 0$, and there is degeneration in the
{\it other} domain $\ogdn$ as $d_{13} \lra 2\sqrt 2$.

Consider the family $\SF$ of curves connecting 
the side $\overline{E_12}$ with the side 
$\overline{E_24}$.
Let us examine the extremal length of this 
family under the pair of limits. In the first case,
as $d_{13} \lra 0$, the two edges $\overline{E_24}$
and $\overline{E_12}$ are becoming disconnected in $\ogup$,
while the domain $\ogdn$ is converging to a non-degenerate
domain. Thus the extremal length of $\SF$ in 
$\ogup$ is becoming infinite, while the corresponding 
extremal length in $\ogdn$ is remaining finite
and positive.  The upshot is that near this limit 
point, $\ext_{\ogup}(\SF) > \ext_{\ogdn}(\SF)$.

Near the other endpoint, where $d_{13}$ is nearly $\sqrt 2$ in $\ogup$,
the opposite inequality holds. This claim
is a bit more subtle, as since the pair of segments
$\overline{23}$ and $\overline{34}$ are converging
to a single point, we see both extremal lengths 
$\ext_{\ogup}(\SF)$ and $\ext_{\ogdn}(\SF)$ are tending
to zero. Yet it is quite easy to compute that the
rates of vanishing are quite different, yielding
$\ext_{\ogup}(\SF) < \ext_{\ogdn}(\SF)$ near this
endpoint.
To see this let $\ogup(\e)$ and $\ogdn(\e)$ denote the
domains $\ogup$ (and $\ogdn$, respectively) with the lengths
of sides
$\ov{23}$ or $\ov{34}$ being $\e$. We are interested in the
Schwarz-Christoffel maps $F_{Gdh,\e}:\BH\to\ogup(\e)$ and
$F_{1/Gdh,\e}:\BH\to\ogdn(\e)$, and in particular at the
preimages $x_2(\e)$, $x_3(\e)$, $x_4(\e)$ (and $y_2(\e)$,
$y_3(\e)$, $y_4(\e)$, resp. ) under $F_{Gdh,\e}$ (and
$F_{1/Gdh}(\e)$, resp.) of the vertices marked $2$, $3$ and $4$. It
is straighforward to see that up to a factor that is bounded away
from both $0$ and $\infty$ as $\e\to0$, 
these positions are
given by the positions of the corresponding pre-images of the
simplified (and symmetric) maps
$$
F_{Gdh,\e}(z) = \int^z t^{-1}(t-x_2)^{-1/2}
(t-x_3)^{1/2} (t-x_4)^{-1/2} dt
$$
and
$$
F_{1/Gdh,\e} = \int^z t^{-1}(t-y_2)^{1/2}
(t-y_3)^{-1/2} (t-y_4)^{1/2} dt
$$

where $\{x_i\}$ and $\{y_i\}$ are bounded away from zero and
we have suppressed the dependence on $\e$ in the
expressions for the vertices. (We can ignore the factor
because we can, for example, normalize the positions of the points
so the the distance $x_4 - x_3$ is the only free parameter.
Once that is done, the factor is determined 
by an integral of a path beginning in the interval 
$[x_4, E_2]$ to a point in the interval $[x_1, E_1]$; in this
situation, both the length of the path and the integrand are bounded 
away from both zero and infinity, proving the assertion.) 
Thus we may compute the asymptotics by setting

\begin{align*}
\e  &= \int^{x_4}_{x_3}\frac1t \sqrt{\f{t-x_3}{(t-x_2)(t-x_4)}}dt\\
&\asymp (x_4 - x_3)^{1/2}\int^1_0 \sqrt{\f s{(s+1)(s-1)}}ds
\end{align*}
using that $s=\f{t-x_3}{x_4-x_3}$, that we have bounded $X_2, x_3,
x_4$ away from zero, and that we have assumed the symmetry
$x_4-x_3=x_3-x_2$.

Thus $x_4(\e)-x_3(\e)=0(\e^2)$.

A similar formal substitution into the integral expression
for $F_{1/Gdh}(\e)$, again using the symmetry of the domain,
yields that

\begin{align*}
\e  &= \int^{y_4}_{y_3}\frac1t \sqrt{\f{(t-y_2)(t-y_4)}{(t-y_3)}}dt\\
&\asymp (y_4 - y_3)^{3/2}\int^1_0\sqrt{\f{(s-1)(s+1)}s}ds
\end{align*}
so $(y_4-y_3)=0(\e^{2/3})$.

As the extremal length $\ext_{\ogup}(\SF)$ and
$\ext_{\ogdn}(\SF)$ of $\SF$ are given by the extremal
lengths in $\BH$ for a family of arcs between intervals that
surround $x_2$, $x_3$, and $x_4$ (or $y_2$, $y_3$, and $y_4$)
in $\BH$, and those extremal lengths are monotone increasing
in the length of the excluded interval $\ov{x_2x_4}$ (or
$\ov{y_2y_4}$, we see from the displayed formulae above  and 
that $\e^{2/3}>\e^2$ for $\e$ small
implies that
\begin{equation}
\label{eqn:genus1}
\ext_{\ogup}(\SF) < \ext_{\ogdn}(\SF)
\end{equation}
for $\e$ small, as desired.

\section{Orthodisks for the Scherk family}
\label{sec3}

In this section, we begin our proof of the existence of
the surfaces $S_g$, the doubly-periodic Scherk surfaces
with $g$ handles. We begin by deciding on the 
form of the relevant orthodisks; our plan is to
adduce these orthodisks from the orthodisks for the
classical Scherk surface $S_0$ and the Karcher
surface $S_1$. It will then turn out that these
orthodisks are quite similar to the orthodisks
we used in \cite{ww1} to prove the
existence of the surfaces $E_g$ of genus
$g$ with one Enneper-like end.

In this section we will introduce 
pairs of orthodisks
and
outline the existence proof for the $S_g$ surfaces,
using the $E_g$ surfaces as the model case.

The existence proof consists of several steps.
The first is to set up a space 
$\Delta=\Delta_g$ of {\it geometric coordinates}
such that each point in this space gives rise to a pair of 
conjugate orthodisks as described in section \ref{sec2}. 

Given such a pair, one canonically obtains a pair of
marked Riemann
surfaces with meromorphic $1$-forms  having complex conjugate periods. If
the surfaces
were conformally equivalent, these two $1$-forms would serve as the
$1$-forms $G dh$ and $\frac1G dh$ in the \wei representation. 

After that, it remains to find a point in the geometric coordinate space so
that the two surfaces are
indeed conformal. To achieve this, 
a nonnegative {\it height function} $\height$
is constructed on the coordinate space with the following 
properties:
\begin{enumerate}
\item  $\height$ is proper;
\item $\height=0$ implies that the two surfaces are conformal;
\item Given a surface $S_{g-1}$,
there is a smooth locus $\SY$ which lies properly in the $S_g$ coordinate
space $\Delta_g$ whose closure contains $S_{g-1} \in \p\overline{\Delta_g}$.  
On that locus $\SY \subset \Delta_g$, if $d \height=0$, then actually $\height=0$.
\end{enumerate}

The height should be considered as some adapted measurement of
the conformal distance between the two surfaces. Hence it is natural to 
construct such a function using conformal invariants. 
We have chosen to build an expression using the extremal lengths of
suitable cycles.

The first condition on the height poses a severe restriction on the choice
of the geometric coordinate system: 
The extremal length of a cycle becomes zero or infinite only if the surface
develops a node near that cycle. 
Hence we must at least ensure
that when reaching the boundary 
$\p\overline{\Delta_g}$ of the geometric coordinate domain $\Delta_g$, 
{\it at least one}  of the two surfaces degenerates {\it conformally}.

This condition is called {\it completeness} of the geometric 
coordinate domain $\Delta_g$.

Fortunately, we can use the definition of the geometric coordinates for
$E_g$ to derive
complete geometric coordinates for $S_g$.

We recall the geometric coordinates that we used in \cite{ww1} to
prove the existence of the Enneper-ended surfaces $E_g$. There,
both domains $\ogup$ and $\ogdn$ were bounded by staircase-like
objects we referred to as 'zigzags': in particular, 
the boundary of a domain was a properly embedded 
arc, which alternated between ($g+1$)
purely vertical segments and
($g+1$) purely horizontal segments and was symmetric across a 
diagonal line. Any such boundary is determined up to
translation by the lengths of its initial $g$ finite-length
sides, and up to homothety by any subset of those of 
size $g-1$. Thus, the geometric coordinates we used for such a 
domain $\ogup$ or $\ogdn$ were the lengths of the first $g-1$
sides. These coordinates are obviously complete.

\begin{remark} We were fortunate in \cite{ww1}, as we will
be in the present case, to be able to restrict our attention
to  orthodisks which embed in the plane.  For orthodisk
systems that branch over the plane (see \cite{ww2})
or are not planar (see \cite{dou1}), the description of the
geometric coordinates can be quite subtle.
\end{remark}

Recall that the orthodisks for the Chen-Gackstatter surfaces of 
higher genus were obtained
by taking the negative $y$-axis and the positive $x$-axis and 
replacing the subarc from $(0,-a)$ to $(a,0)$
by  a monotone arc consisting of horizontal and vertical segments
which were  symmetric with respect to the 
diagonal $y=-x$.
The two regions separated by this `zigzag' constituted a pair of
orthodisks. 
The geometric coordinates were given by the
edge lengths of the finite segments above the diagonal $y=-x$.

For our new surfaces, we continue the above construction as
follows. Denote the vertex of the new subarc that meets the diagonal by $(c,-c)$. Choose $b>c$. We then 
intersect the upper left region 
with the half planes $\{x>-b\}$ and $\{y<b\}$. Similarly, we intersect
the lower right region with the half planes $\{x<b\}$ 
and $\{y>-b\}$. This procedure defines two domains which we denote  by $\ogup$ and $\ogdn$. We use the
convention that $\ogup$ is the domain where the vertex $(c,-c)$ makes a $3\pi/2$ angle.

As geometric coordinates for this pair of orthodisks we take the edge lengths as
before and in addition the width $b$ of the half-infinite  vertical
and horizontal strips.

\begin{theorem}
\label{thm:complete} This coordinate system for $S_g$ is complete.
\end{theorem}
\begin{proof}
Certainly if one of the finite edges degenerates, the conformal 
structure also leaves all compact sets in its moduli space.
Next, if the geometric coordinate $b-c$ tends to 0, the two vertices on the diagonal $y=-x$ coalesce, so that the 
extremal length of the arc connecting
$P_0E_2$ to $E_1P_{2g}$ tends to $\infty$, and so the surface has 
also degenerated.
\end{proof}

Why should such an orthodisk system correspond 
to a doubly-periodic minimal surface of genus $g$? Here we are
both generalizing the intuition given by Karcher's surface,
or alternatively relying on numerical simulation 
(see Figure \ref{fig:genus4}).  Either way, we can conjecture the divisor
data for a fundamental (and planar) piece of the surface 
$S_g$, and use this to define the orthodisk of the surface,
hence the developed image of a fundamental piece.

To formalize the discussion, we introduce:

\begin{definition}
\label{def:reflexive}
A pair of orthodisks is called {\em reflexive} if there is a vertex-
and label-preserving holomorphic map between them.
\end{definition}

Then we have:

\begin{theorem}
Given a reflexive pair of orthodisks of genus $g$, there is a 
doubly-periodic minimal surface 
of genus $g$ in
$T\times \BE$ with two orthogonal top and two bottom ends. 
\end{theorem}

\begin{proof}
We first construct the underlying Riemann surface by taking the
$\ogup$ orthodisk, doubling it along the boundary, and then taking a
double branched cover  of that, branched at the vertices. 
This gives us a Riemann surface $X_g$ of genus $g$.

That Riemann surface carries a natural cone metric 
induced by the flat metric of $\ogup$. As all identifications 
are done by parallel translations, 
this cone metric has trivial holonomy and hence the 
exterior derivative of its developing map defines a 
$1$-form which we call $G dh$. 
This $1$-form is well-defined, up to multiplication by a complex number.

By the reflexitivity condition, the very same Riemann surface $X_g$ 
carries another cone metric, being
induced from the  $\ogdn$ orthodisk and the canonical identification 
of the $\ogup$ and $\ogdn$  orthodisks
by a vertex-preserving conformal diffeomorphism. This second orthodisk 
defines a second $1$-form, denoted
by $\frac1G dh$, also well-defined only up to scaling.

The free scaling parameters are now fixed (up to an arbitrary real
scale 
factor which only affects the size
of the surface in $\BE^3$) so that the developed
$\ogup$ and
$\ogdn$ are truly complex conjugate if we use the same base point and 
base direction for the two developing
maps.

This way we have defined the \wei data $G$ and $dh$ on a Riemann surface $X_g$.

We show next that the resulting minimal surface has the desired 
geometric properties.
The cone points on $X_g$ come only from the orthodisk vertices: 
the finite vertices  $P_j$, being branch points, lift to 
a single cone point 
(also denoted $P_j$). The other finite cone
points $V_+$ and $V_-$ give also only one cone point on the surface, 
denoted by $V$. The half strips lead to
four cone points $E_i$. From the cone angles we can easily deduce the 
divisors of the induced $1$-forms as
\begin{align*}
(G dh)&= P_0^2 P_2^2 \cdots P_{2g}^2 E_1^{-1}E_2^{-1}E_3^{-1}E_4^{-1} \\
(\frac1G dh) &= P_1^2 P_3^2 \cdots P_{2g-1}^2 V^2 E_1^{-1}E_2^{-1}E_3^{-1}E_4^{-1} \\
(dh)&=P_0 \cdots P_{2g} V E_1^{-1}E_2^{-1}E_3^{-1}E_4^{-1}
\end{align*}

These data guarantee that the surface is properly immersed, 
without singularities, and complete.
The points $P_i$  and $V$ correspond to the points with vertical 
normal at the attached handles, while 
the $E_i$ correspond to the four ends.

As $dh$ has only simple zeroes and poles, its periods will all 
have the same phase, and using a local
coordinate it is easy to see that the periods must all be real. 

For the  cycles
in $\ogup$ and $\ogdn$ corresponding to finite edges, 
the conjugacy condition ensures that all of the 
periods are purely real. The cycles around the ends are
similarly conjugate by the construction of the orthodisks. 
The symmetry of the domain ensures that the ends are 
orthogonal.

\end{proof}

\section{Existence Proof: The Height Function}
\label{sec4}

\subsection{Definition and Reflexivity of the Height Function}
\label{sec41}

For a cycle $c$ connecting pairs of edges denote by 
$\ext_{\ogup}(c)$ and $\ext_{\ogdn}(c)$
the extremal lengths of the cycle in the $G dh$ and $\frac1G dh$
orthodisks, respectively. Recall that this makes sense
as we have a natural topological identification of these domains (up to
homotopy) mapping corresponding
vertices onto each other.

The height function on the space of geometric coordinates will
be a sum over several summands of the following type:

\begin{definition} Let $c$ be a cycle. Define
$$
\height(c)=\Big\vert e^{1/\ext_{\ogup}(c)}-e^{1/\ext_{\ogdn}(c)}
\Big\vert^       2+
\Big\vert e^{\ext_{\ogup}(c)}-e^{\ext_{\ogdn}(c)} \Big\vert^2
$$
\end{definition}

The rather complicated shape of this expression is required
to prove the properness of the height function: Because there
are sequences of points in the space of geometric coordinates
which converge to the boundary so that {\it both} orthodisks
degenerate for the same cycles, the above expression must be
very sensitive to different rates with which this happens.

Due to the Monodromy Theorem \ref{thm:mono}, it is
sometimes possible to detect such rate differences in the
growth of $\exp \frac1{\ext(c)}$ for degenerating cycles with
$\ext(c)\to 0$.

The assumptions of the Monodromy Theorem impose certain
restrictions on the choice of cycles for the height, and
there are further restrictions coming from the Regeneration
Lemma \ref{lemma50} below.

Before we introduce the cycles formally, we need to set some
notation. In the figure below, we have labelled the finite
vertices of the staircase for $S_g$ as 
$\{P_0, \dots, P_{2g}\}$, the end vertices as
$E_1$ and $E_2$, and the finite vertices on the outside
boundary components of $\ogup$ and $\ogdn$ as
$V_+$ and $V_-$, respectively. Note that 
in the combined Figure \ref{fig:geomcoord}, the vertices
$P_i$ proceed in a different order for the domain 
$\ogdn$ than they do for the domain $\ogup$.

\begin{figure}[h] 
\centering
\includegraphics[width=5in]{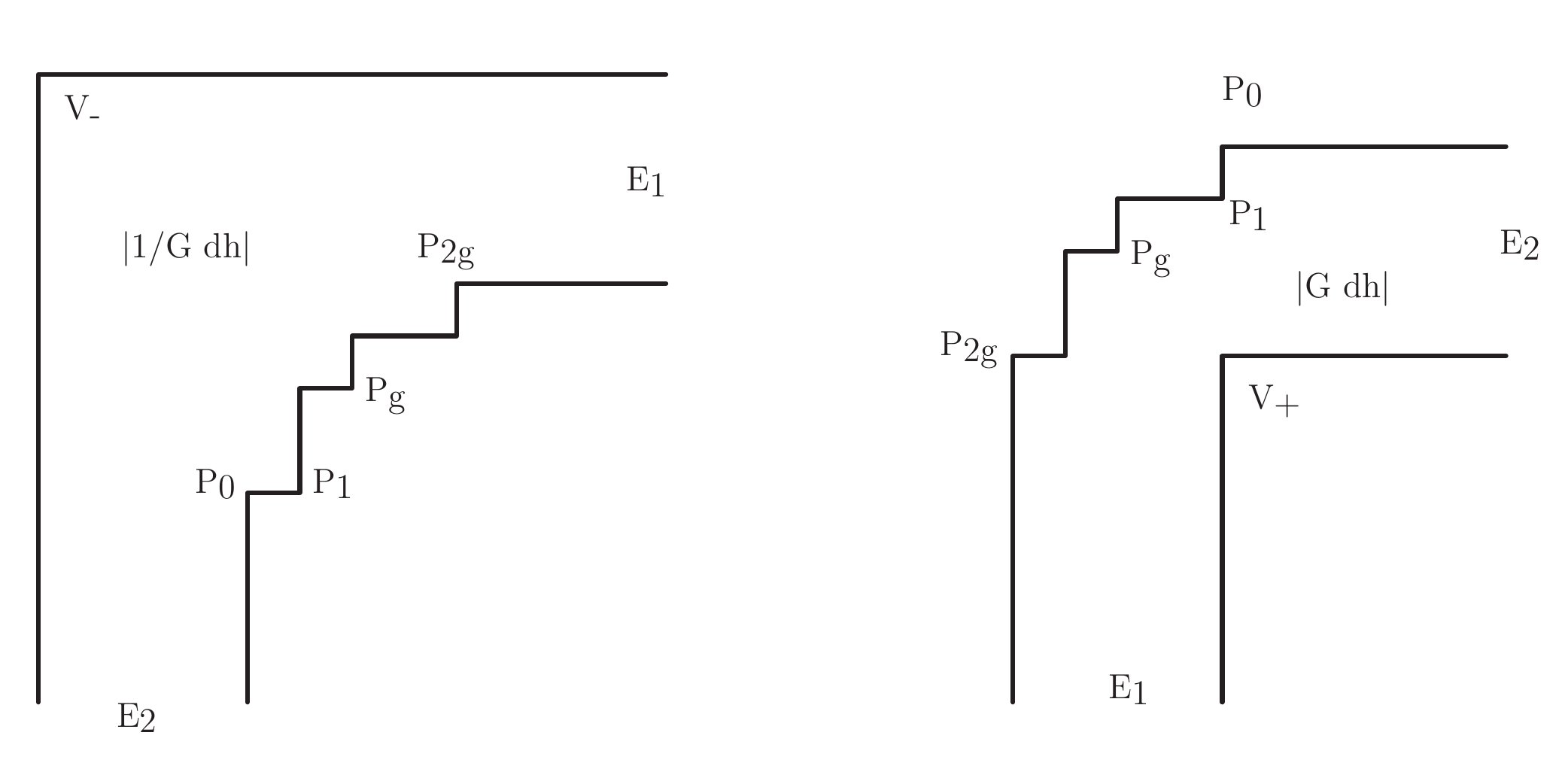} 
\caption{Geometric Coordinates}
\label{fig:geomcoord}
\end{figure}

At this point, there is a difficulty in keeping the
notation consistent; a consistent choice of orientation
of the \ga map $G$ results in the two regions
switching labels as we increase the genus by one; we will
circumvent that notational issue by requiring the \ga
map $G$ to have the orientation for odd genus
opposite to that which is has for even genus -- thus,
the angle at $P_g$ in $\ogup$ will always be $3\pi/2$,
independently of $g$. See Figure~\ref{fig:geomcoord}.

Now let's introduce the cycles formally.

Let $c_i$ denote the cycle in a domain which 
encircles the segment $\overline{P_{i-1}P_{i}}$;
here $i$ ranges from $1$ to $g-1$, and from 
$g+2$ to $2g$. In addition,
let $\delta$ connect the segment $\overline{E_2P_0}$
to the segment $\overline{P_{2g}E_1}$.
This last segment $\delta$ is loosely analogous 
in its design and purpose to the
arc we used in the second proof of the
existence of the Karcher surface $S_1$.

We group these cycles in pairs symmetric with respect to the 
$y=-x$ diagonal and also
require that the cycles are symmetric themselves:

To this end, set
$$
\gamma_i = c_{i} + c_{2g+1-i}, \qquad i=1,\dots,g-1.
$$

These cycles will detect degeneracies on the boundary with many
finite vertices, while $\delta$ detects degeneration of the pair
of boundaries in $\ogup$.

We next use these cycles to define a proper height function
on the moduli space $\Delta_g$ of pairs of orthodisks. Note that $\dim \Delta_g=g$,
so we are using $g$ cycles.

\begin{definition}
\label{def412} The height for the $E_g$ surface is 
defined as
$$
\height=\sum_{i=1}^{g-1} \height(\gamma_i) + 
\height(\delta)
$$
\end{definition}

\begin{lemma}
\label{lemma413} 
If $\height=0$, the two orthodisks are reflexive, i.e. there is a vertex preserving conformal map between them.
\end{lemma}

\begin{proof}
Map the $\ogup$ orthodisk conformally to the upper 
half plane $\BH$ so that $P_g$ is mapped to $0$, and $V_+$ to
$\infty$. As the domain $\ogup$ is symmetric about 
a diagonal line connecting $P_g$ with $V_+$,
our mapping is equivariant
with respect to that symmetry and the
reflection in $\BH$ about the imaginary axis ---
in particular,  $E_1$ is taken
to $1$, while $E_2$ is taken to $-1$. 
The vertices $P_j, E_k, V_{\pm}$ 
are mapped to points $\tilde{P_j}, \tilde{E_k}, \tilde{V_{\pm}}
\in \BR$ and the cycles 
$\g_j$ are carried to
cycles in the upper half plane which are symmetric with respect to
reflection in the $y-$axis.

Now, note that if the height $\height$ vanishes, then so do
each of the terms $\height(\gamma_i)$ and $\height(\delta)$.
Thus the corresponding extremal lengths $\ext_{\ogup}(\G)$
and $\ext_{\ogdn}(\G)$ agree on the curves $\G = \delta, \g_1, \dots,
\g_{g-1}$. It is thus enough to show that that set 
$$\{\ext_{\ogup}(\g_1), \dots  
\ext_{\ogup}(\g_{g-1}), \ext_{\ogup}(\delta)\}$$
of extremal lengths determines the
conformal structure of $\ogup$, or equivalently in this
case of a planar domain $\ogup$, the positions of the
distinguished points
$\{\tilde{P_j}, \tilde{e_k}, \tilde{V}_{\pm}\}$ on the boundary of 
the image $\BH$.

Now, $\tilde{P_0} = -\tilde{P}_{2g}$, and as $\ext_{\ogup}(\delta)$
is monotone in the position of $\tilde{P}_0 = -\tilde{P}_{2g}$
(having fixed $\tilde{e}_1 =1 $ and $\tilde{e}_2
= -1$), we see that $\ext_{\ogup}(\delta)$ determines the position
of $\tilde{P}_0$ and  $\tilde{P}_{2g} = -\tilde{P_0}$.
Next we regard $\tilde{P}_1$ as a variable, with the 
positions of $\tilde{P}_2, \dots, \tilde{P}_{g-1}$ depending
on $\tilde{P}_1$.  The point is that any choice of $\tilde{P}_1$,
together with the datum $\ext_{\ogup}(\g_1)$ uniquely determines
a corresponding position of $\tilde{P}_2$; moreover, as our choice
of $\tilde{P}_1$ tends to $-1$, the correspondingly determined
$\tilde{P}_2$ also tends to $-1$, and as our choice
of $\tilde{P}_1$ tends to $0$, the correspondingly determined
$\tilde{P}_2$ pushes towards $0$. Thus since we know that there is at 
least one choice of points 
$\{\tilde{P}_j, \tilde{E}_k, \tilde{V}_{\pm}\}$ on the boundary of 
the image $\BH$ for which the extremal lengths will agree for
corresponding curve systems, we see there is a range of possible values
in $(-1,0)$ for the position of $\tilde{P}_1$, each uniquely
determining a position of $\tilde{P}_2$ in $(\tilde{P}_1,0)$.  Similarly,
for each of those values $\tilde{P}_1$ and $\tilde{P}_2$, the extremal 
length $\ext_{\ogup}(\g_2)$ uniquely determines a 
value for $\tilde{P}_3$ in $(\tilde{P}_2,0)$. We continue, inductively
using the positions of $\tilde{P}_{j-1}$ and $\tilde{P}_j$
and the datum $\ext_{\ogup}(\g_j)$ to determine $\tilde{P}_{j+1}$.
In the end, we have, for each choice of $\tilde{P}_1$,
a sequence of uniquely determined positions 
$\tilde{P}_2, \dots, \tilde{P}_{g-1}$, with the positions of 
all the determined points depending monotonically on the 
choice of $\tilde{P}_1$. Of course the positions of
$\tilde{P}_{g-3}$, $\tilde{P}_{g-2}$, $\tilde{P}_{g-1}$ and 
$\tilde{P}_g=0$ determine the value $\ext_{\ogup}(\g_{g-1})$,
which is part of the data.  By the monotonicity of the dependence
of the choice of positions $\tilde{P}_2, \dots, \tilde{P}_{g-1}$ 
on the choice of $\tilde{P}_1$, we see that the choice of
$\tilde{P}_1$, and hence all of the values, is uniquely 
determined.

Thus all of the distinguished points on the boundary
of $\BH$ are determined, hence so is the conformal
structure of $\ogup$.
\end{proof}

As we
clearly have that $\SH \ge 0$, we see that our task in the
next few sections is to find zeroes of $\SH$. This we accomplish,
in some sense, by flowing down $-\nabla \SH$ along a 
nice locus $\SY \subset \Delta_g$ avoiding both critical 
points and a neighborhood of $\p\overline{\Delta_g}$.

An essential property of the height is its analyticity:

\begin{proposition}\label{prop:extanalytic}
The height function is a real analytic function on $\Delta_g$.
\end{proposition}

\begin{proof}
The height is an analytic expression in extremal lengths of cycles 
connecting edges of polygons.
That these are real analytic, follows by applying the
Schwarz-Christoffel formula twice: first to map the polygon conformally
to the upper half plane, 
and second to map the upper half plane to a rectangle so that the
edges the cycle 
connects become parallel edges of the rectangle. Then it follows that
the modulus of the rectangle 
depends real analytically on the geometric coordinates of the orthodisks. 
\end{proof}

\subsection {The properness of the height function}
\label{sec42}

\begin{theorem}
\label{thm421} 
The height function is proper on the
space of geometric coordinates. 
\end{theorem}

The proof is based on the following fundamental principle
we have used for the identical purpose in \cite{ww1} and
\cite{ww2}.

\begin{theorem}\label{thm:mono}  
Let $c$ be a cycle as above.
Consider a sequence of  pairs of conjugate orthodisks 
$\ogup$ and $\ogdn$ indexed by a parameter $n$
such that either $c$ encircles an edge shrinking
geometrically to zero and both $\ext_{\ogup}(\gamma)\to 0 $ 
and $\ext_{\ogdn}(\gamma)\to 0$
or $c$ foots on an edge shrinking geometrically to zero and 
both $\ext_{\ogup}(\gamma)\to \infty $ and $\ext_{\ogdn}(\gamma)\to \infty$.
Then $\height(c)\to \infty$ as $n\to\infty$. 
\end{theorem}

We postpone the proof of this theorem until after the 
proof of Theorem \ref{thm421}.

{\bf Proof of Theorem \ref{thm421}.} 
To show that the  height functions from section \ref{sec41} are proper, we need to prove
that for any
sequence of points in $\Delta$ converging to some boundary point, at least one of the terms
in the
height function goes to
infinity. The idea is as follows. By 
the completeness of the geometric coordinate system
(Theorem \ref{thm:complete}), at least one
of the two orthodisks degenerates conformally.
We will now analyze those possible geometric degenerations.

Begin by observing that
we may normalize the geometric coordinates such that the 
boundary of $\ogup$ containing the vertices $\{P_i\}$
has fixed `total length' $1$ 
between $P_0$ and $P_{2g}$, i.e. the sum of the
Euclidean lengths of the finite length edges is 1.
If the geometric degeneration 
involves degeneration
in this outer boundary component of $\p\overline{\ogup}$, 
then one of the cycles $\g_j$ that either
encircles or ends on an edge (or in the case 
where $P_{g-1}, P_g$ and $P_{g+1}$
coalesce, a pair of edges) must shrink to zero. 
By the Monodromy Theorem \ref{thm:mono},
the corresponding term of the 
height function goes to infinity, and we are done.

Alternatively, if there is no geometric degeneration on the
boundary component of $\ogup$ containing the vertices $\{P_i\}$,
then the degeneration must come from the vertex $V_+$ 
either limiting on $P_g$, or tending to infinity.
In the first case,
as in our discussion of the extremal
length geometry behind Karcher's surface, this then 
forces the extremal length $\ext_{\ogup}(\delta)$ to go to
$\infty$, while, in the dual orthodisk, no degeneration is
occuring and $\ext_{\ogdn}(\delta)$ is converging to a positive
value. Naturally, this also sends the corresponding term
$\SH(\delta)$ to $\infty$.

In the latter case of $V_+$ tending to infinity, and
no other degeneration on $\p\overline{\ogup}$, it is 
convenient to adopt a different normalization: for this
case, we set $d(P_g, V_+)=1$. This forces all $V_1,\ldots,V_{2g}$
to coalesce simultaneously. Then the argument
proceeds quite analogously to the argument we gave in 
section \ref{sec3} for the existence of Karcher's surface.  In particular,
the present case follows directly from that case, once
we take into account a well-known background fact.

{\bf Claim:} Let $\Omega \subset \Omega^{'}$, let $\Gamma$
be a curve system for $\Omega$ and let $\Gamma^{'}$
be a curve system for $\Omega^{'}$. Suppose that
$\Gamma \subset \Gamma^{'}$.  Then 
$\ext_{\Omega}(\Gamma) \ge \ext_{\Omega^{'}}(\Gamma^{'})$.

{\bf Proof of Claim:} Any candidate metric $\rho^{'}$ for
$\ext_{\Omega^{'}}(\Gamma^{'})$ restricts to a metric 
$\rho$ for $\ext_{\Omega}(\Gamma)$.
The minimum length of elements of $\Gamma$ 
in this restricted metric is at least as large 
as the minimum length of $\Gamma' \supset \Gamma$ in the extended
metric; moreover, the area of the metric restricted
to $\Omega$ is no larger than that of 
the $\rho'$-area of $\Omega' \supset \Omega$. Thus
$$
\f{\ell^2_\rho(\G)}{\text{Area}(\rho)} \ge 
\f{\ell^2_{\rho^{'}}(\G^{'})}{\text{Area}(\rho^{'})}\ .
$$
The claim 
follows by comparing these ratios for an extremizing
sequence $\rho_n^{'}$ for $\ext_{\Omega^{'}}(\Gamma^{'})$.

Then observe that the orthodisk $\ogup$ for 
$S_g$ sits strictly outside the orthodisk
$\ogup$ for $S_1$, where here we compare
corresponding orthodisks whose first and last vertices
($P_0$ and $P_{2g}$) agree, while $P_g$ for $S_1$ is constructed using the 
existing geometric data.  (See Figure
\ref{fig:compare}.) Thus the extremal length, say
$\ext^g_{\ogup}(\delta)$, for the curve
$\delta$ in the genus $g$ version of the domain 
$\ogup$, is less than the genus one version
$\ext^1_{\ogup}(\delta)$ of the extremal length of 
$\delta$ for that domain, i.e.
$\ext^g_{\ogup}(\delta) \le \ext^1_{\ogup}(\delta)$.

\begin{figure}[h] 
\centering
\includegraphics[width=4.5in]{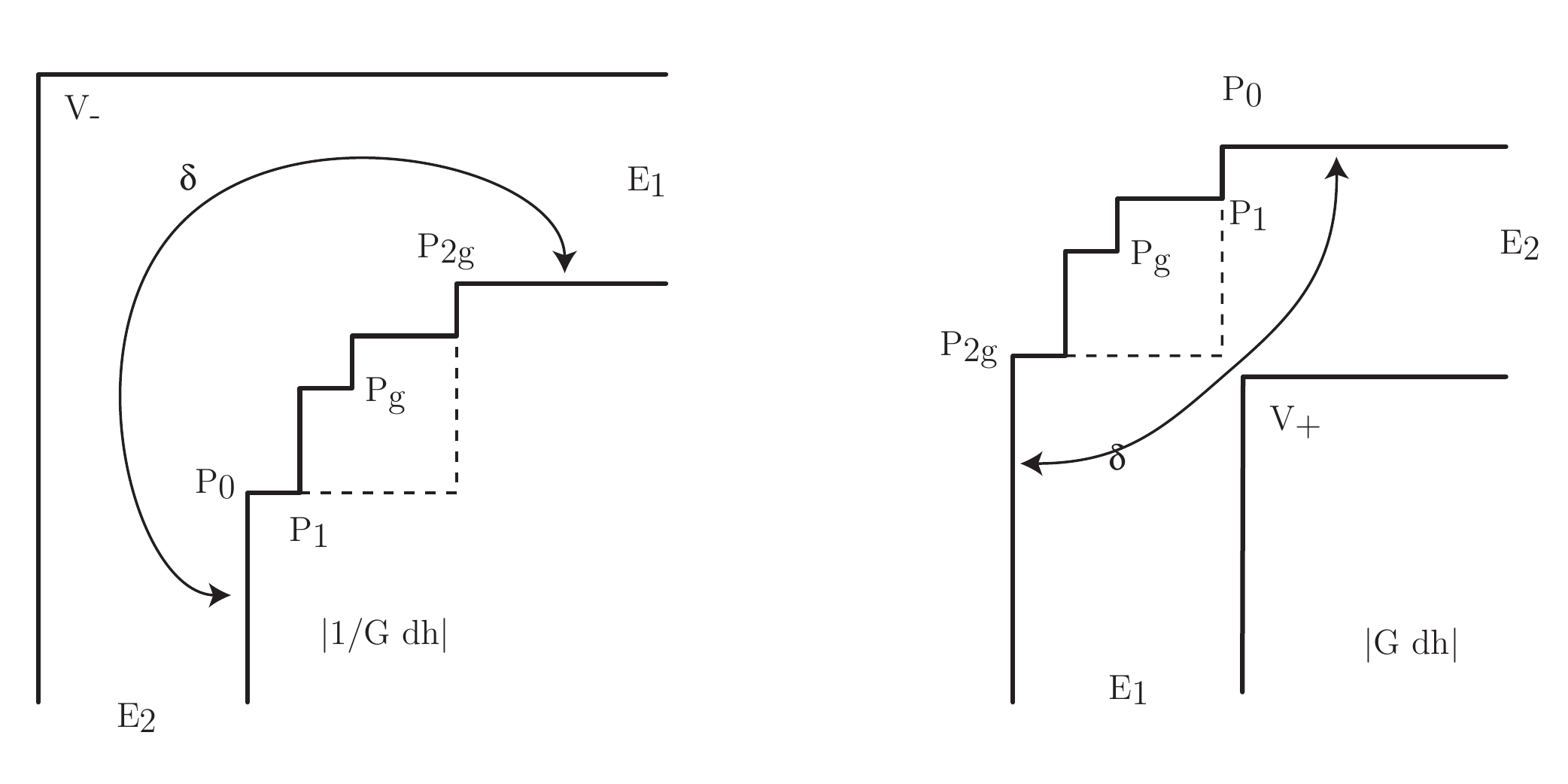} 
\caption{Orthodisk comparison}
\label{fig:compare}
\end{figure}

On the other hand, the corresponding orthodisk $\ogdn$ for 
$S_g$ sits strictly {\it inside} the corresponding
orthodisk $\ogdn$ for $S_1$, using the 
standard correspondence of $\ogup$ and $\ogdn$
orthodisks. 
Observe that for the $V_i$ close enough together,
the vertex $V_1$ of $S_1$ will lie outside that of  $\ogdn$ of $S_g$.
Thus $\ext^1_{\ogdn}(\delta) \le \ext^1_{\ogdn}(\delta)$.

Thus because we have 
$\ext_{\ogdn}^1(\delta) > \ext_{\ogup}^1(\delta)$ for the
case of $S_1$ (see \eqref{eqn:genus1}), with both quantities tending to
zero (at different rates), the claim
implies that we have the analogous inequality
$\ext_{\ogdn}^g(\delta) >> \ext_{\ogup}^g(\delta)$ holding
for $S_g$.  Moreover, the claim (and the notation) also implies that 
$\ext_{\ogup}(\delta)$ tends to zero at a rate
distinct from that of $\ext_{\ogdn}(\delta)$. 
Thus the height function $\height(\delta)$ in such a case 
tends to infinity.

There is one final case to consider, which is hidden a bit
because of our usual choice of conventions: it is only here
that this normalizing of notation can be misleading.  The issue is
that, in Figure~\ref{fig:compare} for instance, the angle at $P_g$
and the angle of $V_-$ are both $\pi/2$ in $\ogup$, and the angles are
$3\pi/2$ at both $P_g$ and $V_+$ in $\ogdn$. However, we of course
need to consider degenerations when the corresponding angles do not
agree, for example when the angle at $P_g$ in $\ogup$ is $3\pi/2$
while the angle at $V_-$ (also) in $\ogup$ is $\pi/2$.  [In that
situation, we will also be in the situation where the angle 
at $P_g$ in $\ogdn$ is $\pi/2$ and the angle at $V_+$ in
$\ogdn$ is $3\pi/2$.]

Now this situation is simply only a bit more complicated than the 
last case we considered, as it follows by applying the claim as
before and then the comparison for genus one,
only this time we have to apply that claim twice before invoking the
comparison for genus one.  

We also use a slightly different auxiliary construction, which we now 
explain. In the situation where the angle of $V_-$ in $\ogup$ is
$\pi/2$ while the angle at $P_g$ in $\ogup$ is
$3\pi/2$, imagine `cutting a notch out of $\ogdn$' near $P_g$: more 
precisely, replace a neighborhood of $\partial\ogup$ near $P_g$ with 
three vertices $P_{g-1}^*, P_g^*, P_{g+1}^*$ and edges between them
that alternate $\pi/2$ and $3\pi/2$ angles in the usual way.  This
creates an orthodisk $\ogup^*$ for a surface of quotient genus $g+1$, where the
angle at $P_g^*$ is now $\pi/2$, now equaling the angle
at $V_-$ opposite $P_g^*$.  Of course, this notch-cutting also
determines a conjugate domain $\ogdn^*$, where the angle at the (new)
central point $P_g^*$ is now $3\pi/2$, also equaling the angle at 
$V_+$ opposite it.  Thus, in considering the domains $\ogup^*$ and 
$\ogdn^*$, we have returned to the third case we just finished
considering. Fortunately, the comparisons between the extremal lengths
on $\ogup$ and $\ogup^*$  and those between $\ogdn$ and $\ogdn^*$
allow for us to conclude that $\ext_{\ogup}(\delta) >>
\ext_{\ogdn}(\delta)$ as follows: 

$$
\alignedat{2}
\ext_{\ogup}(\delta) &\ge \ext_{\ogup}^*(\delta) \qquad&\text{by the claimed principle}\\
& \ge \ext_{\ogdn}^1(\delta)& \text{as in the third case}&\\
& >>  \ext_{\ogup}^1(\delta)& \\
&\ge \ext_{\ogdn}^*(\delta) &\\
&\ge \ext_{\ogdn}(\delta)\qquad.&
\endaligned
$$

This treats the four
possible cases, and the theorem is proven.
\qed

\subsection {A monodromy argument}
\label{sec43}

In this section, we prove that the periods of orthodisks have incompatible
logarithmic
singularities in suitable coordinates and apply this to prove the 
Monodromy Theorem \ref{thm:mono}.
The main idea is that
to study the dependence of extremal lengths
of the geometric coordinates, it is necessary to understand the 
asymptotic dependence
of extremal lengths of the 
degenerating conformal polygons (which is classical and well-known, see
\cite{Oht}), and the asymptotic dependence of the geometric
coordinates of the degenerating conformal
polygons. This dependence is given by Schwarz-Christoffel 
maps which are well-studied
in many special cases. Moreover, it is known that these 
maps possess asymptotic
expansions in logarithmic terms. Instead of computing this 
expansion explicitly for the
two maps we need, we use a 
monodromy argument to
show that the crucial logarithmic terms have a different sign 
for the two expansions.

Let $\Delta_g$ be a geometric coordinate domain of dimension  $g\ge2$,
i.e. a simply connected domain equipped with defining geometric
coordinates for a pair
of orthodisks $\ogup$ and $\ogdn$  as usual.

Suppose
$\gamma$ is a cycle in the underlying conformal polygon 
which joins two non-adjacent edges
$P_1P_2$ with
$Q_1Q_2$. Denote by $R_1$
the vertex before  $Q_1$ and  by $R_2$
the vertex after   $Q_2$ and observe that by assumption, $R_2\ne P_1$ but that
we can possibly have $P_2=R_1$.
Introduce a second
cycle
$\beta$ which connects $R_1Q_1$ with $Q_2R_2$.

The situation is illustrated in the figure below; we have
replaced the labels of $P_i, V_{\pm}$ and $E_j$ that we use
for vertices in $\partial\ogup$ and $\partial\ogdn$ with 
generic labels of distinguished points on the boundary of the 
region:  these will represent in general the situations
that we would encounter in the orthodisk.  Of course, we 
retain the convention of using the same label name for
corresponding vertices in $\partial\ogup$ and $\partial\ogdn$.

\begin{figure}[h] 
\centering
\includegraphics[width=3in]{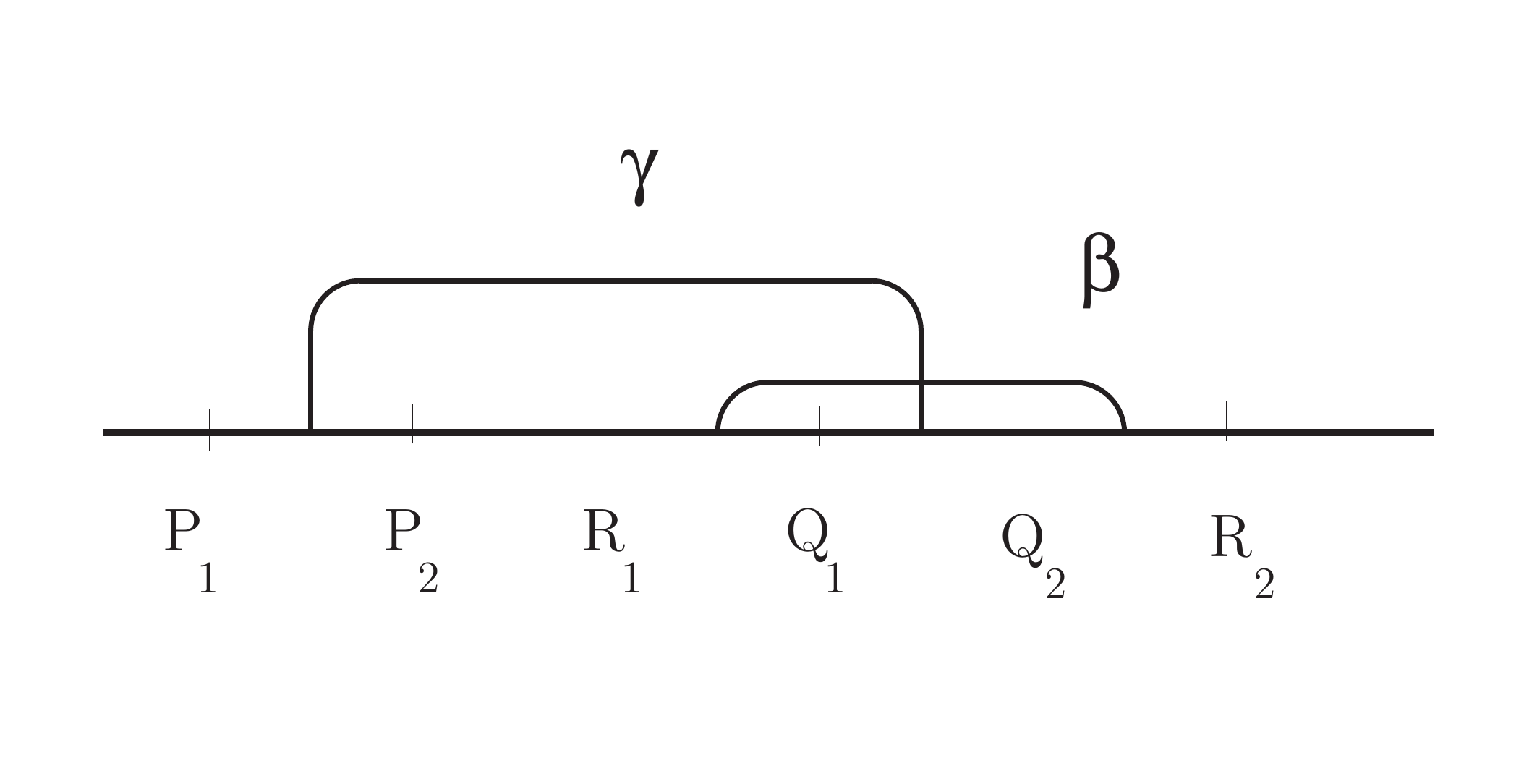} 
\caption{Monodromy argument}
\label{fig:mono1}
\end{figure}

We formulate the claim of Theorem \ref{thm:mono} more precisely in the following two
lemmas:

\begin{lemma}
\label{lem:mono1} 
Suppose that for a sequence $p_n\in \Delta$ with
$p_n\to p_0\in \partial\Delta$ we have that
$\ext_{\ogup(p_n)}(\gamma) \to 0$ and $\ext_{\ogdn(p_n)}(\gamma)\to 0$. Suppose furthermore that $\gamma$
is a cycle encircling an edge which degenerates geometrically to $0$ as $n\to\infty$.
Then
$$
\left\vert e^{1/\ext_{\ogup(p_n)}(\gamma)}-e^{1/\ext_{\ogdn(p_n)}(\gamma)} \right\vert^
	2\to\infty
	$$
\end{lemma}

\begin{lemma}
\label{lem:mono2} Suppose that for a sequence $p_n\in \Delta$ with
$p_n\to p_0\in \partial\Delta$ we have that
$\ext_{\ogup(p_n)}(\gamma) \to \infty$ and $\ext_{\ogdn(p_n)}(\gamma)\to \infty$.
Suppose furthermore that $\gamma$
is a cycle footing on an edge which degenerates geometrically to $0$ as $n\to\infty$.
Then
$$
\left\vert e^{\ext_{\ogup(p_n)}(\gamma)}-e^{\ext_{\ogdn(p_n)}(\gamma)} \right\vert^
	2\to\infty
	$$
\end{lemma}

\begin{proof} We first prove Lemma \ref{lem:mono1}.

Consider the conformal polygons corresponding to the pair of orthodisks.
Normalize the punctures by M\"obius transformations so that
$$P_1=-\infty, P_2=0, Q_1=\epsilon, Q_2=1$$
for $\ogup$ and
$$P_1=-\infty, P_2=0, Q_1=\epsilon', Q_2=1$$
for $\ogdn$.

If $\alpha$ is a curve in a domain $\Omega \subset \BC$, then define 
$\per \alpha(\Omega) = \int_{\alpha}dz$.  Here our focus is on periods of the
one-form
$dz$ as we are typically interested in domains $\Omega$ which are
developed images of pairs $(\Omega, \omega)$ of domains and 
one-forms on those domains, i.e. $z(p) = \int_{p_0}^p \omega$.
By the assumption of Lemma \ref{lem:mono1}, we know that $\epsilon, \epsilon'\to 0$
as $n\to\infty$. 

We now allow $Q_1$ to move in the complex plane and apply the 
Real Analyticity Alternative Lemma \ref{lemma435} below to the curve
$\epsilon = \epsilon_0
e^{it}$: here we are regarding the position of $Q_1$ as traveling
along a small circle around the origin, i.e. its 
defined position $\e \in \BR$ has been extended to allow
complex values. We will
  conclude from that lemma  that either 

\begin{equation}
\frac{\vert\per\gamma(\ogup)\vert}{\vert\per\beta(\ogup)\vert} +
\frac1\pi \log\epsilon =: F_1(\epsilon)
\label{eqn:alt1}
\end{equation}
is single-valued in $\epsilon$ and 
\begin{equation}
\frac{\vert\per\gamma(\ogdn)\vert}{\vert\per\beta(\ogdn)\vert} -
\frac1\pi \log\epsilon' =:F_2(\epsilon') = F_2(\epsilon'(\epsilon))
\label{eqn:alt2}
\end{equation}
is single-valued in $\epsilon'$ or vice versa, with signs exchanged. Without loss of
generality, we can treat
the first case.

Now suppose that $\epsilon'$ is real analytic (and hence single-valued)
in $\epsilon$ and comparable to $\epsilon$ near $\epsilon=0$. Then using that
$\ogup$ and $\ogdn$ are conjugate implies that
$$\frac{\vert\per\gamma(\ogup)\vert}{\vert\per\beta(\ogup)\vert}=\frac{\vert\per\gamma(\ogdn)\vert}
{\vert\per\beta(\ogdn)\vert}\ .$$
By subtracting the function $F_1(\epsilon)$ in 
\ref{eqn:alt1} from the function $F_2(\epsilon')$ in \ref{eqn:alt2}
(both
of which are single-valued in $\epsilon$) we get that
$$\log(\epsilon\epsilon'(\epsilon))$$
is single-valued in $\epsilon$ near $\epsilon=0$ which contradicts that
$\epsilon, \epsilon' \to 0$.

Now  Ohtsuka's extremal length formula states that for the current normalization of $\ogup(p_n)$
we have 
$$\ext(\gamma) = O\left(\log\vert \epsilon\vert\right)$$
(see \cite{Oht}). We conclude
that
$$\vert e^{1/\ext_{\ogup(p_n)}(\gamma)}-
e^{1/\ext_{\ogdn(p_n)}(\gamma)}\vert =O\left(\frac1{\epsilon}-\frac1{\epsilon'}\right)
$$
which goes to infinity,
since we have shown that $\e$ and $\e'$ tend to zero at 
different rates.  This proves Lemma \ref{lem:mono1}.

The proof of Lemma \ref{lem:mono2} is very similar: For convenience, we normalize the
points of
the punctured disks such that

$$P_1=-\infty, P_2=0, Q_1=1, Q_2=1+ \epsilon$$
for $\ogup$ and
$$P_1=-\infty, P_2=0,Q_1=1, Q_2=1+\epsilon'$$
for $\ogdn$.

By the assumption of Lemma \ref{lem:mono2}, we know that $\epsilon, \epsilon'\to 0$
as $n\to\infty$. We now apply the Real Analyticity Alternative 
Lemma \ref{lemma435} below to the curve
$1+\epsilon_0
e^{it}$ and conclude that
$$\frac{\per\gamma(\ogup)}{\per\beta(\ogup)} + \frac1\pi \log\epsilon$$
is single-valued in $\epsilon$ while
$$\frac{\per\gamma(\ogdn)}{\per\beta(\ogdn)} - \frac1\pi \log\epsilon'$$
is single-valued in $\epsilon'$.
The rest of the proof is identical to the proof of Lemma \ref{lem:mono1}.
\end{proof}

To prove the needed  Real Analyticity Alternative Lemma \ref{lemma435}, 
we need asymptotic expansions of the extremal length in 
terms of the geometric coordinates of
the orthodisks. Though not much is known explicitly about extremal lengths
in general,
for the chosen cycles we can reduce this problem to an asymptotic control of
Schwarz-Christoffel integrals. Their monodromy properties allow us to
distinguish their
asymptotic behavior by the sign of logarithmic terms.

We introduce some notation:
suppose we have an orthodisk such that the angles at the vertices
alternate between $\pi/2$ and $-\pi/2$ modulo $2\pi$. 
(We will
also allow some angles to be $0$ modulo $2\pi$ but they will not be 
relevant for this argument.) Consider the
Schwarz-Christoffel map
$$
F:z \mapsto \int^z (t-t_1)^{a_1/2}\cdot\ldots\cdot(t-t_n)^{a_n/2} \,dt
$$
from a conformal polygon with vertices at $t_i$ to this orthodisk:
here the exponents $a_j$ alternate between $-1$ and $+1$, depending 
on whether the angles at the vertices are $\pi/2$ or $-\pi/2, \pmod
{2\pi}$, respectively. Choose
four distinct vertices
$t_{i},t_{i+1}, t_j, t_{j+1}$ (not necessarily consecutive). 
Introduce a cycle $\gamma$ in the upper half
plane connecting edge
$(t_{i},t_{i+1})$ with edge $(t_j, t_{j+1})$ and denote by $\bar\gamma$ the
closed cycle obtained from
$\gamma$ and its mirror image at the real axis. Similarly, denote by
$\beta$ the cycle connecting
$(t_{j-1},t_j)$ with $(t_{j+1},t_{j+2})$ and by $\bar\beta$ the cycle
together with its mirror image.

Now consider the Schwarz-Christoffel period integrals
\begin{align*}
F(\gamma) &= \frac12 \int_{\bar\gamma}
(t-t_1)^{a_1/2}\cdot\ldots\cdot(t-t_n)^{a_n/2} dt\\
F(\beta) &= \frac12 \int_{\bar\beta}
(t-t_1)^{a_1/2}\cdot\ldots\cdot(t-t_n)^{a_n/2} dt
\end{align*}
as  multivalued functions depending on the
{\it now complex} parameters $t_i$.

\begin{lemma}\label{lemma433} Under analytic continuation of $t_{j+1}$ around
$t_j$ the periods change their values like

\begin{align*}
F(\gamma) &\rightarrow F(\gamma) + 2 F(\beta) \\
F(\beta) &\rightarrow F(\beta )
\end{align*}
\end{lemma}

\begin{proof} 
The path of analytic continuation of $t_{j+1}$ around
$t_j$ gives rise to an isotopy of $\BC$ which moves $t_{j+1}$ along this path.
This isotopy drags 
$\beta$ and
$\gamma$  to new cycles $\beta'$ and $\gamma'$.

Because the curve $\beta$ is defined to surround $t_j$ and $t_{j+1}$, the
analytic continuation merely returns $\beta$ to $\beta'$. Thus,
because $\beta'$ equals $\beta$, their periods are also equal. 
On the other hand, the curve $\g$ is not equal to $\g'$:
informally, $\g'$ is obtained as the Dehn twist of $\g$ around
$\bar\beta$. Now,
the period of $\gamma'$ is obtained by developing the flat  structure of the
doubled orthodisk along $\gamma'$. To compute
this flat structure, observe 
the crucial fact that the angles at the orthodisk 
vertices are either $\pi/2$ or
$-\pi/2$, modulo $2\pi$. In either case, 
we see from the developed flat structure
that the period of $\gamma'$ equals the period of $\gamma$ plus twice
the period of $\beta$.
\end{proof}

Now denote by $\delta:= t_{j+1}-t_j  $ and fix all $t_i$ other than
$t_{j+1}$: we regard $t_{j+1}$ as the independent variable, here
viewed
as complex, since we are allowing it to travel around $t_j$.

\begin{lemma}[Analyticity Lemma]\label{lemma424} The function
$F(\gamma)-\frac{\log\delta}{\pi i} F(\beta)$ is single-valued and
holomorphic in $\delta$ in a neighborhood of
$\delta=0$.
\end{lemma}

\begin{proof}  By definition,the function is 
locally holomorphic in a punctured neighborhood of $\delta=0$. By Lemma \ref{lemma433}
it extends to be single valued in a (full) neighborhood of $\delta=0$. 
\end{proof}

We will now specialize this picture to the situation at hand --- an orthodisk
where $\gamma$ represents one of the distinguished cycles $\gamma_i$. Then $F(\gamma)$ and $F(\beta)$ are either real or imaginary, and are perpendicular. 
Thus Lemma \ref{lemma424} implies that 
$|F(\gamma)|\pm \frac{\log\delta}{\pi } |F(\beta)|$
is real analytic in $\delta$ with {\em one} choice of sign. The crucial observation is now that whatever alternative holds,
the {\em opposite} alternative will hold for the {\em conjugate} orthodisk. More precisely: 

Let $F_1$ and $F_2$ be the Schwarz-Christoffel maps associated to a pair of conjugate orthodisks.
These will be defined on different but consistently labeled punctured upper half planes. Let $\delta_i$ refer to the complex parameter $\delta$ introduced above for the maps $F_i$, respectively. Then 
\begin{lemma}[Real Analyticity Alternative Lemma]\label{lemma435}

Either $| F_1(\gamma)|  -
\frac{\log\delta_1}\pi |
F_1(\beta) |$ or $| F_1(\gamma)|  + \frac{\log\delta_1}\pi
|
F_1(\beta)|$ is real analytic in $\delta_1$ for $\delta_1=0$. In the first
case,
$| F_2(\gamma)|  + \frac{\log\delta_2}\pi |
F_2(\beta) |$ is real analytic in $\delta_2$, while in the second case,
$| F_2(\gamma)|  - \frac{\log\delta}\pi |
F_2(\beta) |$ is real-analytic in $\delta_2$.
\end{lemma}

\begin{proof} We have already noted that either alternative holds in both cases.
It remains to show that it holds with opposite signs.
For some special values $\delta_1,\delta_2>0$, the two orthodisks are
conjugate. 
For instance, we can assume
that for these values, $F_1(\gamma)=F_2(\gamma)>0$. 
Then $F_1(\beta)$ and $F_2(\beta)$ are both imaginary with opposite signs,
and the claimed alternative holds for these values of $\delta_1,
\delta_2$. 
By continuity, the alternative holds for all $\delta_1$ and $\delta_2$.
\end{proof}

\begin{remark}
A concrete way of understanding the phenomenon here is that the
asymptotic expansion of the period of a curve meeting a degenerating 
cycle $\beta$, where the edge for $\beta$ has preimages
$b$ and $b+\epsilon$, has a term of the form $\pm\epsilon^k \log
\epsilon$,
where the sign relates to the geometry of the orthodisk.  
\end{remark}

\section{The Flow to a Solution}
\label{sec5}

The last part of the proof of the Main Theorem requires us to
prove the

\begin{lemma}[Regeneration Lemma]\label{lemma50} There is, for a given genus
$g$, a certain (good) locus $\SY \subset \Delta_g$
in the space $\Delta_g$ of geometric coordinates for with the
following properties:
\begin{itemize}
\item $\SY$ lies properly within the space of geometric
coordinates;
\item if $d \height=0$ at a point on the locus $\SY$, then actually
$\height=0$ at that point.
\end{itemize}
\end{lemma}

This locus will be defined  by the requirement that all but one
of the extremal lengths of the distinguished cycles 
of the $Gdh$ and $\frac1G dh$ orthodisks are equal.

\subsection{Overall Strategy}
\label{sec:strategy}

In this section we continue the proof of the existence 
of the surfaces $\{S_g\}$.  In the previous
sections, we defined an associated
moduli space $\Delta = \Delta_{g}$ of
pairs of conformal structures $\{\ogup, \ogdn\}$ equipped with
geometric coordinates $\vec{\bold t}=(t_i,...,t_g)$.

We defined
a height function $\height$ on the moduli space $\Delta$ and proved that
it was a proper function: as a result, there is a critical point 
for the height function in $\Delta $, and our overall goal in the 
next pair of sections is a proof that this critical point represents
a reflexive orthodisk system in $\Delta $, and hence, by 
our fundamental translation of the period problem for 
minimal surfaces into a conformal equivalence problem, a minimal
surface of the form $S_g$.
Our goal in the present section is a description of the tangent
space to the the moduli space $\Delta $: we wish to display how  
infinitesimal changes in the geometric coordinates
$\vec{\bold t}$ affect the height function.  In particular, it
would certainly be sufficient for our purposes to prove the statement

\begin{model}\label{model511} If $\vec{\bold t_0}$ is not a reflexive
orthodisk system, then there is an element $V$ of the tangent space
$T_{\vec{\bold t_0}} \Delta $ for which $D_V\height \ne 0$.
\end{model}

This would then have the effect of proving that our critical point
for the height function is reflexive, concluding the existence 
parts of the proofs of the main theorem.

We do not know how to prove or disprove this model statement in its
full generality.  On the other hand, it is not necessary for the 
proofs of the main theorems that we do so. Instead we will replace 
this statement by a pair of lemmas.

\begin{lemma}\label{lemma512} Let $\SY \subset \Delta$ 
is a real one-dimensional subspace of $\Delta $ which is defined 
by the equations 
$\height(\gamma_1)=\height(\gamma_2)=...=\height(\gamma_{g-2})=\height(\delta)=0$.
If $\vec{\bold t_0} \in \SY$ has positive
height, i.e. $\height(\vec{\bold t_0})>0$,
then there is an element $V$ of the tangent space
$T_{\vec{\bold t_0}}\SY$ for which $D_V\height \ne 0$.
\end{lemma}

\begin{lemma}\label{lemma513} There
is an analytic subspace $\SY \subset \Delta = \Delta_g$, for which 
$\SY = \{\height(\gamma_1)=\height(\gamma_2)=...=\height(\gamma_{g-2})
=\height(\delta)=0\}$.
\end{lemma}

Given these lemmas, the proof of the existence 
of a pair of conformal orthodisks $\{\ogup, \ogdn\}$ is
straightforward.

\begin{proof}{Proof of Existence of Reflexive Orthodisks.} Consider the locus
$\SY$ guaranteed by Lemma \ref{lemma513}. By Theorem~\ref{thm421}, the height 
function $\height$ is proper on $\SY$, so the height function
$\height|_\SY$ has a critical point (on $\SY$). By Lemma \ref{lemma512}, this
critical point represents a point of $\height=0$, i.e a reflexive
orthodisk by Lemma \ref{lemma413}.
\end{proof}

The proof of Lemma \ref{lemma512} occupies the current section while the
proof of Lemma \ref{lemma513} is given in the following section.

{\it Remarks on deformations of conjugate pairs of orthodisks.}
Let us discuss informally the proof of Lemma \ref{lemma512}.
Because angles of corresponding vertices in the
$\ogup\leftrightarrow \ogdn$ correspondence sum to $0 \pmod{2\pi}$, the
orthodisks fit together along corresponding
edges, so conjugacy of orthodisks requires corresponding edges to move
in different directions: if the  edge $E$ on $\ogup$
moves ``out,'' the corresponding edge $E^*$ on $\ogdn$
edge moves ``in'', and vice versa (see the figures below). Thus 
we expect that if $\g$ has an endpoint on $E$, then
one of the extremal lengths of $\gamma$ decreases,
while the other extremal length of $\gamma$ on the other orthodisk 
would increase: this will force the height $\height(\gamma)$ of
$\gamma$ to have a definite sign, as desired.  This is the 
intuition behind Lemma \ref{lemma512}; a rigorous argument requires us
to actually compute derivatives of relevant extremal lengths
using the formula \ref{tag22}. We do this by displaying, fairly explicitly, 
the deformations of the orthodisks (in local coordinates on $\ogup/\ogdn$)
as well as the differentials of extremal lengths, also in 
coordinates. After some preliminary notational description in 
section \ref{sec52}, we do most of the computing in section \ref{sec53}.  Also in section
\ref{sec53} is the key technical lemma, which relates the formalism 
of formula \ref{tag22}, together with the local coordinate descriptions
of its terms, to the intuition we just described.

\subsection{Infinitesimal pushes}
\label{sec52}

We need to formalize the previous discussion. As always we are
concerned with relating the Euclidean geometry of the 
orthodisks (which corresponds directly with the periods of the
\wei data) to the conformal data of the domains $\ogup$ and $\ogdn$.
From the discussion above, it is clear that the allowable 
infinitesimal motions in $\Delta$, which are parametrized in
terms of the Euclidean geometry of 
$\ogup$ and $\ogdn$,  are given by infinitesimal
changes in lengths of finite sides,
with the changes being done simultaneously on $\ogup$ and $\ogdn$
to preserve conjugacy. The link to the conformal geometry is
the formula  \ref{tag22}: a motion which infinitesimally transforms 
$\ogup$, say, will produce an infinitesimal change in the 
conformal structure. This tangent vector to the moduli space
of conformal structures is represented by a Beltrami differential.
Later, formula  \ref{tag22} will be used, together with knowledge of the
cotangent vectors $ d \ext_{\ogup}(\cdot)$ and $ d \ext_{\ogdn}(\cdot)$,
to determine the derivatives of the relevant extremal lengths, hence
the derivative of the height.

To begin, we explicitly compute the effect of infinitesimal pushes of 
certain edges on the extremal lengths
of relevant cycles. 
This is done by explicitly displaying the infinitesimal deformation
and then using this formula to compute the sign of the derivative
of the extremal lengths, using formula \ref{tag22}.
There will be two 
different cases to consider.

\begin{enumerate}
\item[Case A.] Finite {\it non-central} edges 
of the type $P_{i-1}P_i$
for $i<g$.
\item[Case B.] An edge (finite or infinite)
and its symmetric side meet in a
corner, for instance $P_{g-1}P_g$. 
\end{enumerate}

For each case there are two subcases, which we can describe as
depending on whether the given sides are horizontal or vertical.
The distinction is, surprisingly, a bit important, as together
with the fact that we do our deformations in pairs, it provides 
for an important cancelation of (possibly) singular terms in 
Lemma \ref{lemma521}. We defer this point for later, while here we 
begin to calculate the relevant Beltrami differentials in the 
cases.

While logically it is 
conceivable that
each infinitesimal motion might require two different types
of cases, depending on whether the edge we are deforming 
on $\ogup$ corresponds on $\ogdn$ to an edge of the same type
or a different type, in fact this issue does not
arise for the particular case of the Scherk surfaces we
are discussing in this paper.  By contrast, it does arise
for the generalized Costa surfaces we discussed in
\cite{ww2}.

Case A. Here the computations are quite analogous to those that we found in 
\cite{ww1}; they differ only in orientation of the boundary of the orthodisk.
We include them for the completeness of the exposition.

\begin{figure}[h] 
\centering
\includegraphics[width=3.5in]{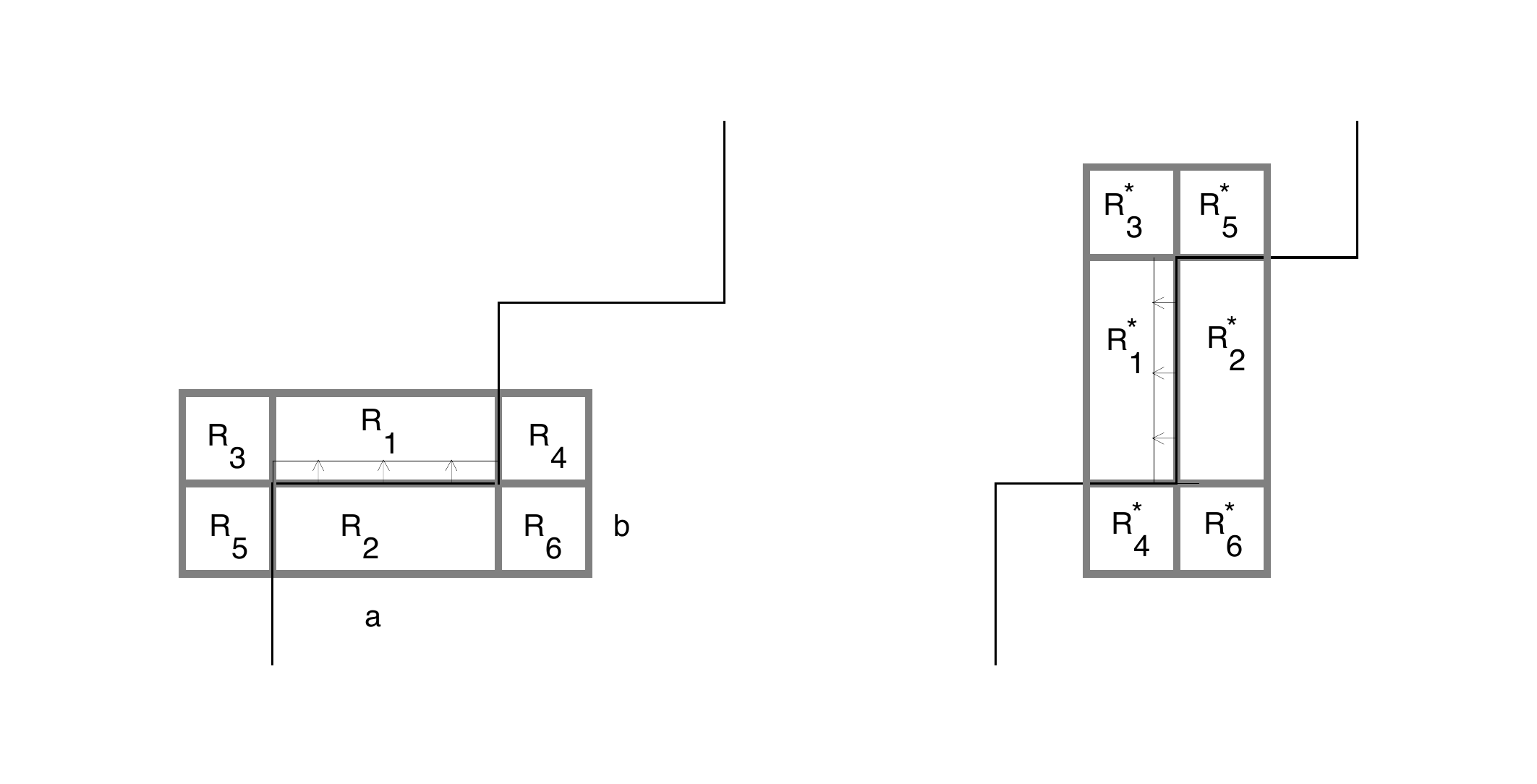} 
\caption{Beltrami differential computation --- case A }
\label{fig:casea}
\end{figure}

We first consider the case of a horizontal finite side; as in the
figure above, we see that the
neighborhood of the horizontal side of the orthodisk in the plane
naturally divides into six regions which we label $R_1$,...,$R_6$.
Our deformation $f_\e = f_{\e, b, \delta}$
differs from the identity only in such a neighborhood, 
and in each of the six regions, the map
is affine.
In fact we have a two-parameter family of these deformations, all
of which have the same infinitesimal effect, with the parameters
$b$ and $\delta$ depending on the dimensions of the supporting
neighborhood.

\begin{equation}
\label{tag51a}
f_\e(x,y) = \begin{cases}
\left(x,\e +\f{b-\e}b y\right),   &\{-a\le x\le a, 0\le y\le b\}=R_1\\
\left(x,\e +\f{b+\e}b y\right),   &\{-a\le x\le a, -b\le y\le0\}=R_2\\
\left(x,y+\f{\e+\f{b-\e}b y-y} \delta(x+ \delta +a)\right),   &\{-a-\delta\le x\le-a,
0\le y\le b\} = R_3\\
\left(x,y-\f{\e+\f{b-\e}b y-y} \delta(x-\delta-a)\right),   &\{a\le x\le a+ \delta,
0\le y\le b\} = R_4\\
\left(x,y+\f{\e+\f{b+\e}b y-y} \delta(x+ \delta +a)\right),   &\{-a-\delta\le x\le-a,
-b\le y\le0\} = R_5\\
\left(x,y-\f{\e+\f{b+\e}b y-y} \delta(x-\delta-a)\right),   &\{a\le x\le a+ \delta,
-b\le y\le0\} = R_6\\
(x,y)   &\text{otherwise}
\end{cases}
\end{equation}
where we have defined the regions $R_1,\dots,R_6$ within the
definition of $f_\e$. Also note that here the orthodisk contains the arc
$\{(-a,y)\mid0\le y\le b\}\cup\{(x,0)\mid-a\le x\le a\}\cup
\{(a,y)\mid-b\le y\le0\}$.
Let $E$ denote the edge being pushed, defined above as $[-a,a]\times\{0\}$.

Of course $f_\e$ differs from the identity only on a neighborhood
of the edge $E$, so that $f_\e$ takes the symmetric orthodisk to
an asymmetric orthodisk. 
We next modify $f_\e$ in a neighborhood of the reflected
(across the $y=-x$ line) segment $E^*$ in an analogous way with a
map $f^*_\e$ so that $f^*_\e\circ f_\e$ will preserve
the symmetry of the orthodisk. 

Our present conventions are that the edge $E$ is horizontal; this forces
$E^*$ to be vertical and we now write down $f^*_\e$ for such a
vertical segment; this is a straightforward extension of the
description of $f_\e$ for a horizontal side, but we present the
definition of $f^*_\e$ anyway, as we are crucially interested in
the signs of the terms. So set
\begin{equation}
\label{tag51b}
f^*_\e = \begin{cases}
\left(-\e +\f{b-\e}bx, y\right),   &\{-b\le x\le0, -a\le y\le a\}=R^*_1\\
\left(-\e +\f{b+\e}bx, y\right),   &\{0\le x\le b, -a\le y\le a\}=R^*_2\\
\left(x - \f{-\e+\f{b-\e}b x-x} \delta(y-\delta-a), y\right),   &\{-b\le x\le0, a\le
y\le a+ \delta\}=R^*_3\\
\left(x + \f{-\e+\f{b-\e}b x-x} \delta(y+ \delta +a), y\right),   &\{-b\le x\le0,
-a-\delta\le y\le-a\}=R^*_4\\
\left(x - \f{-\e+\f{b+\e}b x-x} \delta(y-\delta-a), y\right),   &\{0\le x\le b, a\le
y\le a+ \delta\}=R^*_5\\
\left(x + \f{-\e+\f{b+\e}b x -x} \delta(y+ \delta +a), y\right),   &\{0\le x\le b, -a-\delta\le
y\le-a\}=R^*_6\\
(x,y)   &\text{otherwise}
\end{cases}
\end{equation}
Note that under the reflection across the line $\{y=-x\}$, the
region $R_i$ gets taken to the region $R_i^*$.

Let $\nu_\e = \f{\left(f_\e\right)_{\bar z}}{\left(f_\e\right)_z}$
denote the Beltrami differential of $f_\e$, and set
$\dot\nu=\f d{d\e}\bigm|_{\e=0}\nu_\e$. Similarly, let $\nu^*_\e$
denote the Beltrami differential of $f^*_\e$, and set $\dot\nu^*=\f
d{d\e}\bigm|_{\e=0}\nu^*_\e$. Let $\dot\mu=\dot\nu+\dot\nu^*$. Now
$\dot\mu$ is a Beltrami differential supported in a bounded domain
in one of the domains $\ogup$ or $\ogdn$. 
We begin by observing that it is easy to compute that
$\dot\nu=[\f d{d\e}\bigm|_{\e=0}\left(f_\e\right)]_{\bar z}$ 
evaluates near $E$
to

\begin{equation}
\label{tag52a}
\dot\nu = \begin{cases}
\f1{2b},   &z\in R_1\\
-\f1{2b},   &z\in R_2\\
\f1{2b}[x+ \delta +a]/\delta +i\left(1-y/b\right)\f1{2 \delta}=\f1{2b \delta}(\bar z+ \delta +a+ib),
&z\in R_3\\
-\f1{2b}[x-\delta-a]/\delta-i\left(1-y/b\right)\f1{2 \delta}=\f1{2b \delta}(-\bar z+ \delta +a-ib),
&z\in R_4\\
-\f1{2b}[x+ \delta d+a]/\delta +i\left(1+y/b\right)\f1{2 \delta}=\f1{2b \delta}(-\bar z-\delta-a+ib),
&z\in R_5\\
\f1{2b}[x-\delta-a]/\delta-i\left(1+y/b\right)\f1{2 \delta}=\f1{2b \delta}(\bar z-\delta-a-ib),
&z\in R_6\\
0  &z\notin\supp(f_\e-\id)
\end{cases}
\end{equation}

We further compute
\begin{equation}
\label{tag52b}
\dot\nu^* = \begin{cases}
-\f1{2b},   &R^*_1\\
\f1{2b},   &R^*_2\\
\f1{2b \delta}(i\bar z-\delta-a+bi)    &R^*_3\\
\f1{2b \delta}(-i\bar z-\delta-a-bi)   &R^*_4\\
\f1{2b \delta}(-i\bar z+ \delta +a+bi)   &R^*_5\\
\f1{2b \delta}(i\bar z+ \delta +a-bi)    &R^*_6
\end{cases}
\end{equation}

Case B. We have separated this case out for purely expositional 
reasons. We can imagine that the infinitesimal push that 
moves the pair of consecutive sides along the symmetry line
$\{y=-x\}$ is the result of a composition of a pair of 
pushes from Case A, i.e. our diffeomorphism $F_{\e;b, \delta}$
can be written  $F_{\e;b, \delta}= f_{\e} \circ f_{\e}^*$,
where the maps differ from the identity in the {\it union}
of the supports of $\dot\nu_{b, \delta}$ and 
$\dot\nu_{b, \delta}^*$.

It is an easy consequence of the chain rule applied to this 
formula for  $F_{\e;b, \delta}$ that the infinitesimal Beltrami differential
for this deformation is the sum $\dot\nu_{b, \delta} + \dot\nu_{b, \delta}^*$
of the
infinitesimal Beltrami differentials 
$\dot\nu_{b, \delta}$ and 
$\dot\nu_{b, \delta}^*$ defined in formulae \ref{tag52a}, \ref{tag52b}
for Case A (even in a neighborhood of the vertex along the
diagonal where the supports of the differentials $\dot\nu_{b, \delta}$ and 
$\dot\nu_{b, \delta}^*$ coincide).

\subsection{Derivatives of Extremal Lengths}
\label{sec53}

In this section, we combine the computations of $\dot\nu_{b, \delta}$
with formula \ref{tag22} (and its background in section \ref{sec2})
and some easy observations on the nature of the quadratic
differentials $\Phi_{\mu}= \frac12 d \ext_{(\cd)}(\mu)\bigm|_{\cd}$ 
to compute the derivatives of extremal lengths under our
infinitesimal deformations of edge lengths.

We begin by recalling some background from section \ref{sec2}. If we are given
a curve $\gamma$, the extremal length of that curve on an orthodisk,
say $\ogup$, is a real-valued
$C^1$ function on the moduli space of that orthodisk.
Its differential is then a holomorphic quadratic differential
$\Phi_{\gamma}= \frac12d\ext_{(\cd)}(\gamma)\bigm|_{\ogup}$
on that orthodisk; the horizontal foliation of $\Phi_{\g}$ consists of
curves which connect the same edges in $\ogup$ as $\g$,
since $\Phi_{\g}$ is obtained as the pullback of the quadratic
differential $dz^2$ from a rectangle where $\g$ connects
the opposite vertical sides.
We compute the derivative of the 
extremal length function using formula \ref{tag22}, i.e.

$$
\left(d\ext_{\cd}(\gamma)\bigm|_{\ogup}\right)[\nu] =
4\re\int_{\ogup}\Phi_\gamma\nu
$$

It is here where we find that we can actually compute the 
sign of the derivative of the extremal lengths, hence the height
function, but also encounter a subtle technical problem. The point is
that we will discover that
just the topology of the curve $\gamma$ on $\ogup$
will determine the sign of the derivative on an edge $E$, so we will be able to
evaluate the sign of the integral above, if we shrink the support of the
Beltrami differential $\dot\nu_{b,\delta}$ to the edge by sending
$b, \delta $ to zero.  (In particular, the sign of 
$\Phi_\gamma$ depends precisely on whether the foliation of $\Phi= \Phi_\gamma$
is parallel or perpendicular to $E$, and on whether $E$ is
horizontal or vertical.) We then need to know two things:
1) that this limit exists, and
2) that we may know its sign via examination of the sign of 
$\dot\nu_{b, \delta}$ and $\Phi_\gamma$ on the edge $E$. We phrase this as

\begin{lemma}
\label{lemma521} 
{\rm(1)} $\lim_{b\to0, \delta\to0}\re\int\Phi\dot\nu$
exists, is finite and non-zero. {\rm(2)} The (horizontal) foliation of $\Phi = \Phi_{\g}$ is
either parallel or orthogonal to the segment which is
$\lim_{b\to0, \delta\to0}(\supp\dot\nu)$, and {\rm(3)} The expression
$\Psi\dot\nu$ has a constant sign on the that segment $E$, and the
integral \ref{tag22} also has that (same) sign.
\end{lemma}

Of course, in the statement of the lemma, the horizontal
foliation of the holomorphic quadratic differential
$\Phi=\Phi_\gamma$ has regular curves parallel to $\g$.

This lemma provides the rigorous foundation for the intuition
described in the final paragraph of strategy section \ref{sec:strategy}.

\subsection{Proof of the Technical Lemma \ref{lemma521}}
\label{sec523}

\begin{proof}
Let $S_{Gdh}$ denote the double of $\ogup$ across the boundary; the 
metric space $S_{Gdh}$ is a flat sphere with conical singularities,
two of which are metric cylinders. 

The foliation
of $\Phi$, on say $\ogup$, lifts to a foliation on the punctured
sphere, 
symmetric about the reflection about the equator. This
proves the second statement. The third statement follows from the
first
(and from the above discussion
of the topology of the vertical foliation of $\Phi_{\g}$), 
once we prove that there is no infinitude 
of $\int \Phi\dot{\nu}$ as $b,\delta \to 0$ coming from either
the neighborhood of infinity of the infinite edges or the regions $R_3$
and $R_4$ for the finite vertices. This finiteness will
follow from the proof of the first statement. Thus, we are left to prove the
first statement which requires us once again to consider the 
cases A and B.

Case A:  Suppose $\g$ connects two non-central
finite edges $E'$ and $E''$ on
$\ogup$.
To understand the singular behavior of $\Phi= \Phi_{\g}$ near a vertex of
the orthodisk, say $\ogup$, we begin by observing 
(by formula \ref{tag22}) that on a preimage on
$S_{Gdh}$ of such a vertex, the lifted
quadratic differential, say $\Psi$,  has a simple
pole. 
This is consistent with the nature of the foliation of $\Psi$, whose
non-singular horizontal leaves are all freely homotopic 
to the lift of $\g$; the fact itself follows from following
the lift of the canonical quadratic differential on a rectangle.
Thus the singular leaves of $\Psi$
are segments on the equator of the sphere connecting lifts
of endpoints of the edges $E'$ and $E''$.

Now let $\om$ be a local
uniformizing parameter near the preimage of the 
vertex on $S_{Gdh}$ and $\z$ a local uniformizing
parameter near the vertex of $\ogup$ on $\BC$. 
There are two cases to
consider, depending on whether the angle in $\ogup$ at the vertex
is  $3\pi/2$ or $\pi/2$. In the first case, the map from
$\ogup$ to a lift of $\ogup$ in $S_{Gdh}$ is given in coordinates
by $\om=(i\z)^{2/3}$, and in the second case by $\om=\z^2$. Thus, in
the first case we write $\Psi=c\f{d\om^2}\om$ so that $\Phi
=-4/9c(i\z)^{-4/3}d\z^2$, and in the second case we write $\Phi
=4cd\z^2$; in both cases, the constant $c$ is real with sign
determined by the direction of the foliation.

With these expansions for $\Phi$, we can compute
$\lim_{b\to0, \delta\to0}\re\int\Phi\dot\nu$.

Clearly, as $b+ \delta\to0$, as $|\dot\nu|=O\left(\max\left(\f1b,\f1 \delta\right)\right)$, we need
only concern ourselves with the contribution to the integrals of the
singularity at the vertices of $\ogup$ with angle $3\pi/2$.

To begin this analysis, recall that we have assumed that the edge
$E$ is
horizontal so that $\ogup$ has a vertex angle of $3\pi/2$ at the
vertex,
say $P$. This means that $\ogup$ also has a vertex angle of $3\pi/2$ at the
reflected vertex, say $P^*$,  on $E^*$. 
It is convenient to rotate a neighborhood of
$E^*$ through an angle of $-\pi/2$ so that the support of
$\dot\nu$ is a reflection of the support of $\dot\nu^*$ (see
equation \ref{tag51a} through a vertical line. If the coordinates of
$\supp\dot\nu$ and $\supp\dot\nu^*$ are $z$ and $z^*$, respectively
(with $z(P)=z^*(P^*)=0$), then the maps which lift
neighborhoods of $P$ and $P^*$, respectively, to the
sphere $S_{Gdh}$ are given by
$$z\mapsto(iz)^{2/3}=\om \quad\hbox{ and}\quad z^*\mapsto(z^*)^{2/3}=\om^* .$$
Now the
poles on $S_{Gdh}$ have coefficients $c\f{d\om^2}\om$ and
$-c\f{d\om^{*2}}{\om^*}$,
respectively, so we find that when we pull back these poles from
$S_{Gdh}$ to $\ogup$, we have
$\Phi(z)=-\f49c\ {dz^2}/{\om^2}$ while
$\Phi(z^*)=-\f49c\ dz^2/(\om^*)^2$ in the coordinates $z$
and $z^*$ for $\supp\dot\nu$ and $\supp\dot\nu^*$, respectively.
But by tracing through the conformal maps $z\mapsto\om\mapsto\om^2$ on
$\supp\dot\nu$ and $z^*\mapsto\om^*\mapsto(\om^*)^2$, we see that
if $z^*$ is the reflection of $z$ through a line, then
$$\f1{(\om(z))^2}=1/{\ov{\om^*(z^*)^2}}$$
so that the coefficients
$\Phi(z)$ and $\Phi(z^*)$ of
$\Phi=\Phi(z)dz^2$ near $P$ and of
$\Phi(z^*)dz^{*2}$ near $P^*$ satisfy
$\Phi(z)=\ov{\Phi(z^*)}$, at least for the singular
part of the coefficient.

On the other hand, we can also compute a relationship between the
Beltrami coefficients $\dot\nu(z)$ and $\dot\nu^*(z^*)$ (in the
obvious notation) after we observe that $f_\e^*(z^*)=-\ov{f_\e(z)}$.
Differentiating, we find that
\[
\begin{split}
\dot\nu^*(z^*)  &= \dot f^*(z^*)_{\ov{z^*}}\\
&=-\ov{\dot f(z)}_{\ov{z^*}}\\
&=(\ov{\dot f(z)})_z\\
&= \ov{\dot f(z)_{\ov z}}\\
&= \ov{\dot\nu(z)}.
\end{split}
\]
Combining our computations of $\Phi(z^*)$ and $\dot\nu(z^*)$
and using that the reflection $z\mapsto z^*$ {\it reverses} orientation,
we find that (in the coordinates $z^*=x^*+iy^*$ and $z=x+iy$) for
small neighborhoods $N_{\k}(P)$ and $N_{\k}(P^*)$ of $P$
and $P^*$ respectively,
\[
\begin{split}
\re &\int\lm_{\supp\dot\nu\cap N_\k(P)}\Phi(z)
\dot\nu(z)dxdy +
\re\int\lm_{\supp\dot\nu^*\cap N_\k(P^*)}\Phi(z^*)
\dot\nu(z^*)dx^*dy^*\\
&= \re\int\lm_{\supp\dot\nu\cap N_\k(P)}\Phi(z)\dot\nu(z) -
\Phi(z^*)\dot\nu(z^*)dxdy\\
&= \re\int\lm_{\supp\dot\nu\cap N_\k}\Phi(z)\dot\nu(z) -
\ov{[\Phi(z)+O(1)]}\ \ov{\dot\nu(z)} dxdy\\
&= O(b+ \delta)
\end{split}
\]
the last part following from the singular coefficients summing to a
purely imaginary term while $\dot\nu=O\left(\f1b+\f1 \delta\right)$,
and the neighborhood has area $b\delta$. This
concludes the proof of the lemma for this case.

Case B. Here we need only consider the singularities resulting at
the origin, as we treated the other singularities
in Case A.  The lemma in this case follows
from a pair of observations.  First, because of the symmetry across
the line through the vertex under discussion, the differential
$\Psi$ (the lift of $\Phi$) on the sphere is holomorphic, and so the
behavior of $\Phi$ near the vertex is at least as regular as in the
previous cases.  Moreover, because the infinitesimal
Beltrami differential in this case is the sum of infinitesimal Beltrami
differentials encountered in the previous cases A, 
the arguments there on the cancellation of the apparent singularities
of the sum $\dot\nu_{b, \delta} + \dot\nu_{b, \delta}^*$
continue to hold here for the single singularity.

This concludes the proof of Lemma \ref{lemma521}.
\end{proof}

\begin{proof}{Conclusion of the proof of Lemma \ref{lemma512}}. 
Conjugacy of the domains $\ogup$ and $\ogdn$ 
allows that the there is a Euclidean motion which glues the 
domains  $\ogup$ and $\ogdn$ together through identifying the
side $P_iP_{i+1}$ with $P_{2g-i-1}P_{2g-i}$: this is evident
from the construction and is illustrated in
Figure~\ref{fig:geomcoord}.
Thus, if we push an edge $E \subset \partial\overline\ogup$ 
into the domain $\ogup$, we will change the
Euclidean geometry of that domain in ways that will force us
to push the corresponding edge $E^* \subset \partial\overline\ogdn$ 
out of the domain $\ogdn$.

Now, given this geometry of the glued complex
$D=\ogup \cup \ogdn$, we observe that we can 
reduce the term $\height(\g_{g-1})$ 
of the height function $\height$ by an 
infinitesimal push on the
edges meeting the boundary of 
$\g_{g-1}$.  Moreover, because the rest of the terms
of the height function $\height$ vanish along the locus $\SY$ 
to second order in the
deformation variable, we see that any deformation of the orthodisk
will not alter (infinitesimally) the contribution of these terms to $\height$.
Thus the only effect of an infinitesimal deformation of an orthodisk
system on $\SY$ to the height function $\height$ is to the
term $\height(\g_{g-1})$, which is non-zero to
first order by Lemma \ref{lemma521}. 
This concludes the proof of the lemma.

\end{proof}

\subsection{Regeneration}

In the previous section we showed how we might reduce the height
function $\height$ at a critical
point of a locus $\SY$, where the locus  $\SY$ was 
defined as the null locus
of all but one of the heights $\height(\gamma_{g-1})$. 
In this section, we prove Lemma \ref{lemma513}, which 
guarantees the existence of such a locus $\SY$.

Let us review the context for this argument. Basically,
we will prove the existence of the genus $g$ Scherk surface,
$S_g$, by using the existence of the genus
$g-1$ Scherk surface, $S_{g-1}$, to imply
the existence of a locus $\SY \subset \Delta_g$ --- the
Lemma \ref{lemma512} and the Properness Theorem \ref{thm421} then prove the
existence of $S_g$. 

Indeed, our proof of the main theorem is by induction: we make the 

{\bf Inductive Assumption A:} There exists a genus $g-1$ 
Scherk surface $S_{g-1}$.

Thus, all of our surfaces are produced from only slightly less
complicated surfaces; this is the general principle of 'handle
addition' referred to in the title.

For concreteness and ease of notation, we will prove the existence
of $S_3$ assuming the existence of $S_2$. The general
case follows with only more notation. Thus, our present goal is the proof of 

\begin{theorem}\label{theorem61} There is a reflexive orthodisk system
for the configuration $S_3$.
\end{theorem}

\begin{proof} Let us use the given height 
$\height_{3}$ for $S_3$ and consider how the
height $\height_{2}$ for $S_2$ relates to it, 
near a solution for the genus $2$
problem. 

Our notation is given in section \ref{sec41} and is recorded in 
the diagrams below: for instance,
the curve system $\delta $ connects the edges $E_2P_0$ and $P_{6}E_1$.

\begin{figure}[h] 
\centering
\includegraphics[width=3.5in]{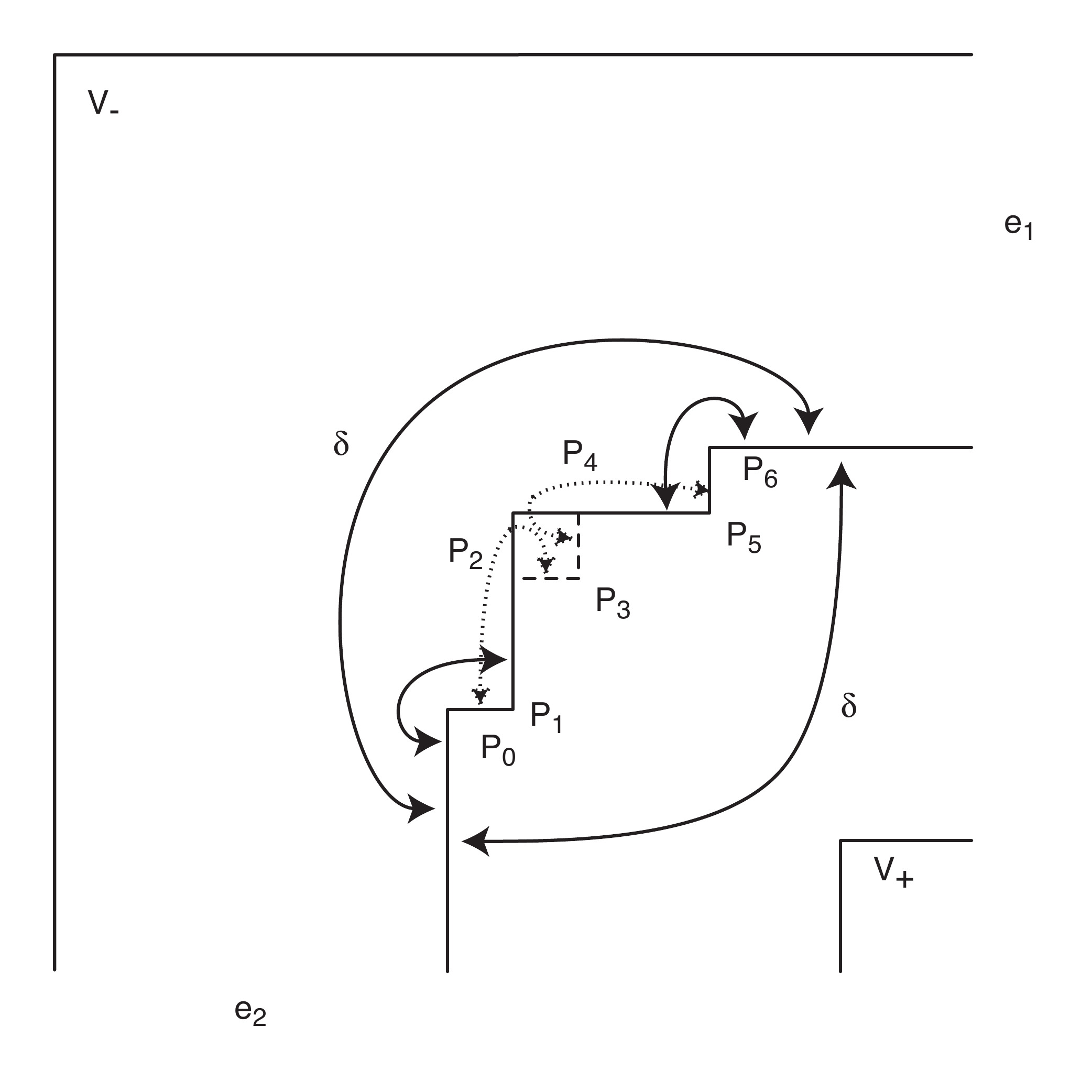} 
\caption{Curve system used for regeneration}
\label{fig:reg}
\end{figure}

We are interested in how an orthodisk system might degenerate.
One such degeneration is shown in the next figure, where the 
points $P_2$, $P_3$, and $P_4$ have coalesced.
The point is that the degenerating family of (pairs
of) Riemann surfaces in $\Delta_{3}$ limits on (a pair of) surfaces
with nodes. (We recall that a surface with nodes is a complex space
where every point has a neighborhood complex isomorphic to
either the disk $\{|z|<1\}$ or a pair of disks 
$\{(z,w)| zw=0\}$ in $\BC^2$.)
In the case of the surfaces corresponding to $\ogup$ and $\ogdn$, the
components of the noded surface (i.e. the regular components of the
noded surface in the complement of the nodes)
are difficult to observe, as the flat structures on the
thrice-punctured sphere components are 
simply single points.

An important issue in this section is that some of our curves
cross the pinching locus on the surface, i.e. the curve on the 
surface which is being collapsed to form the node. In particular,
in the diagram, the dotted curves $\g_2$ are such curves, so their
depiction in the degenerated figure is, well, degenerate: the curves
connect a point and an edge.

\begin{figure}[h] \centering
\includegraphics[width=3.5in]{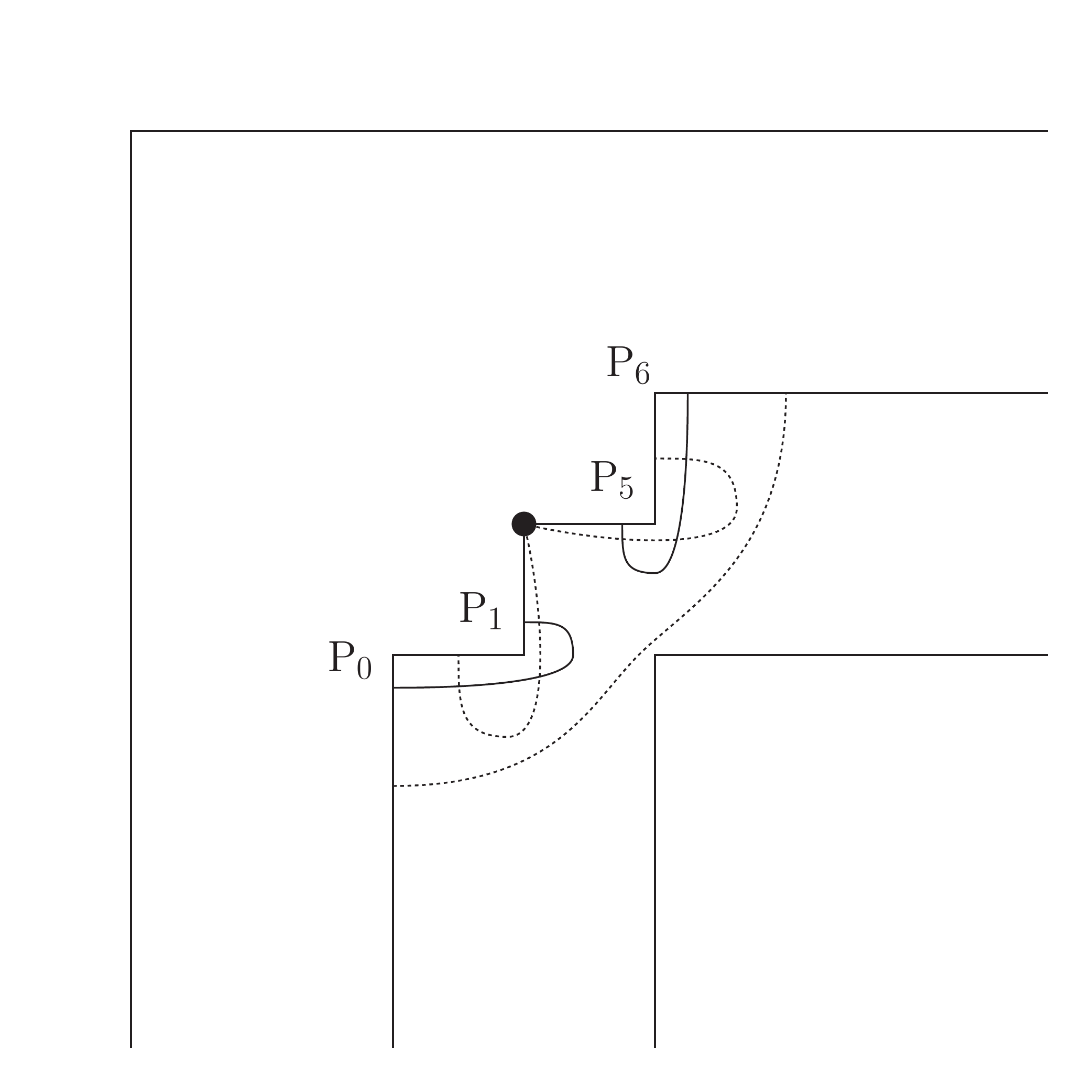} 
\caption{Regenerating orthodisks for the Scherk surfaces}
\label{fig:reg2}
\end{figure}

Note that when we degenerate, we are left with
the orthodisks for the surface of one lower genus, in this case that of
$S_2$.

Our basic approach is to work backwards from this understanding of
degeneration --- we aim to ``regenerate'' the locus $\SY$ in 
$\Delta_3$ from the solution $X_2 \in \Delta_{2} \subset \partial\bar{\Delta_3}$.

We focus on the curves $\delta $ and $\g_1$, ignoring the degenerate
curve $\g_2$. 

(In the general case for $\Delta_{g}$, there are $g-1$
non-degenerate curves 
$\{ \delta, \g_1,..., \g_{g-2}\}$), 
and one degenerate
curve $\g_{g-1}$.)

We restate Lemma \ref{lemma513} in terms of the present (simpler) notation.

\begin{lemma}\label{Lemma 6.2 (Regeneration)}: There is a one-dimensional
analytic closed locus $\SY\subset\ov{\Delta_g}$ so that both
$\ext_{\ogup}(\g_i)=\ext_{\ogdn}(\g_i)$ for $i=1,\dots, g-2$ 
and $\ext_{\ogup}(\delta)=\ext_{\ogdn}(\delta)$
on $\SY$, and $\SY$ is proper
in $\Delta_g$.
\end{lemma}

\begin{proof}  We again 
continue with the notation for $g=3$.

As putatively defined in the statement of the
lemma, $\SY$ would be clearly closed, and would have non-empty
intersection with $\ov{\Delta_2}$ as $\ov{\Delta_2}$ contains the
solution $S_2$ to the genus 2 problem.

We parametrize $\ov{\Delta_3}$ near $X_2$ as
$\Delta_2\x[0,\e)$ and consider the map
$$
\Phi:(X,t): \Delta_2\x[0,\e)\lra\BR^2
$$
given by
$$
(X,t)\mapsto(\ext_{\ogup}(\delta)-\ext_{\ogdn}(\delta),
\ext_{\ogup}(\g_1)-\ext_{\ogdn}(\g_1)).
$$
Here, the coordinate $t$ refers to a specific choice
of normalized geometric coordinate, i.e.
$t=\im(P_0P_1\lra P_2P_3) = \re(P_3P_4\lra P_5P_6)$,
where the periods $(P_0P_1\lra P_2P_3)$ and  $(P_3P_4\lra P_5P_6)$
are measured on the domain $\ogup$. 
In terms of these coordinates, we 
note that whenever either $t=0$, we are
in a boundary stratum of $\ov{\Delta_3}$. The locus
$\{ t>0\} \subset \ov{\Delta_3}$ is a neighborhood
in $Int(\Delta_3)$ with $X_2$ in its closure.

Note that $\Phi(X_2,0)=0$ as $X_2$ is reflexive.

Now, to find the locus $\SY$, we apply the implicit function theorem.
The implicit function theorem says that if

(i) the map $\Phi$ is differentiable, and

(ii) the differential $d\Phi\bigm|_{T_{X_2}\Delta_2}$ 
is an isomorphism onto $\BR^2$,

\noindent
then there exists a
differentiable family $\SY\subset\ov{\Delta_3}$ for which
$\Phi\bigm|_\SY\equiv0$.

We now prove the differentiability
(condition (i)) of $\Phi$: as the locus of
$\ov{\Delta_2}\in\wh{\ov\SM_2} \x \wh{\ov\SM_2}$ is differentiable (here
$\wh{\ov\SM_2}$ refers to a smooth cover of the relevant neighborhood
of $S_2\subset\ov\SM_2$, where $\ov\SM_2$ is the Deligne-Mostow
compactification of the moduli space of curves of genus
two) the theorem of Gardiner-Masur [GM]
implies that $\Phi$
is differentiable, as we have been very careful to choose curves
$\{\delta, \g_1\}$ which are non-degenerate in a neighborhood
of $\ov{\Delta_2}$ near the genus two solution $S_2$,
with both staying in a single regular component of the noded surface.

We are left to treat (ii), the invertibility of the 
differential $d\Phi\bigm|_{T_{S_2}\Delta_2}$.
To show that $d\Phi\bigm|_{T_{S_2}\Delta_2}$ is an isomorphism, we
simply prove that it has no kernel. To see this, choose a
tangent direction in $T_{S_2}\Delta_2$ interpreted 
as a perturbation of the geometric
coordinates for $S_2$. To be concrete, we might fix the 
distance between 
the parallel semi-infinite sides
and vary the finite lengths (or periods) of the sides 
$P_iP_{i+1}$. 
Now, up to
replacing the infinitesimal variation with its negative, 
one of the finite-length edges has moved into the
interior of
$\ogup$ as in this case, with our normalization, the only edges
free to move are those finite edges. Connect each
of those positively moving edges with a curve system from
$E_2P_0$ (and, symmetrically, $P_4E_1=P_{2g}E_1$).
The result is a large curve system (say $\Gamma$)
consisting of 
classes of curves from (possibly) several free homotopy classes.
In addition, let $\nu$ be the associated Beltrami differential
to this variation, as in section \ref{sec52}.

The flow computations in section \ref{sec5}  then say that this entire curve system 
$\Gamma$ has
extremal length in $\ogup$ which has decreased by an
amount proportional to $|\nu|$ while the extremal length on $\ogdn$ has
increased by an
amount proportional $|\nu|$ on $\ogdn$. Thus, in any set of differentiable coordinates for the Teichmuller
space of the surfaces in $\Delta_2$, the difference of coordinates for 
$\ogup $and $\ogdn$ is by an amount proportional to $|\nu|$.
Thus $d\Phi(\nu) \ge c|\nu|$ 
(for $c>0$) which proves the assertion.

To finish the proof of the lemma, we need to show that
$\SY\bigm|_{\Delta_3}$ is an analytic
submanifold of $T_3\x T_3$, where $T_3$ is the \tec space of 
genus three curves: this follows {\it on the interior of}
$\Delta_3$ from the fact that Ohtsuka's formulas for extremal length
are analytic and the map from $\Delta_2$ to extremal lengths has
non-vanishing derivative.

This concludes the proof of Lemma \ref{Lemma 6.2 (Regeneration)} for the case $g=3$ and hence
also
the proof of Theorem \ref{theorem61}. We have already noted that the argument
is completely general, despite our having presented it in the concrete
case of $S_3$; thus, by adding more notation, we have proven
Lemma \ref{Lemma 6.2 (Regeneration)} in full generality. Naturally, this also completes the
proof of Lemma \ref{lemma513}.

\end{proof}

\section{Embeddedness of the doubly-periodic  Scherk surface with handles }
\label{sec6}

In this section, we follow Karcher
\cite{ka4} and use the conjugate surface method from section~\ref{sec211}
to prove

\begin{proposition}\label{Proposition 6.1}
The doubly-periodic Scherk surface with handles ($S_g$) is
embedded.
\end{proposition}

We  consider a quarter of the surface, as defined by the shaded lower
left quarter square of the fundamental domain of the torus, as
in Figure \ref{fig:s1divs}.

Of course,
this is also the part of the surface used to define the
orthodisks $\ogup$ and $\ogdn$. This surface patch, say $\Sigma_g$,
is bounded by planar
symmetry curves and one vertical end and is contained within the infinite
box over a 'black' checkerboard square. It will be sufficient to show
that this patch $\Sigma_g$
is embedded, as the rest of the surface $S_g$ in space
is obtained by reflecting the image $\Sigma_g$ of this quarter surface 
across vertical planes.
To show that the image of the quarter surface is embedded,
we prove that the conjugate surface is a graph over a 
convex domain; the result then follows by Krust's Theorem
in section \ref{sec211}.  Thus we prove

\begin{lemma}\label{lemma62} The conjugate surface for $\Sigma_g$ is
a minimal graph over a convex domain.
\end{lemma}

We apply the basic principles that straight lines and planar
symmetry curves get interchanged by the operation of conjugation
of minimal surfaces, and angles get preserved by this operation.
Using these principles, we compute the conjugate surface for
$\Sigma_g$.

We assert that the conjugate
surface has the form depicted in the figure \ref{fig:conjsch1} 
  and described below. 
The surface in the figure extends horizontally along the positive horizontal coordinate
directions to infinity.

\begin{figure}[h] 
\centering
\includegraphics[width=2in]{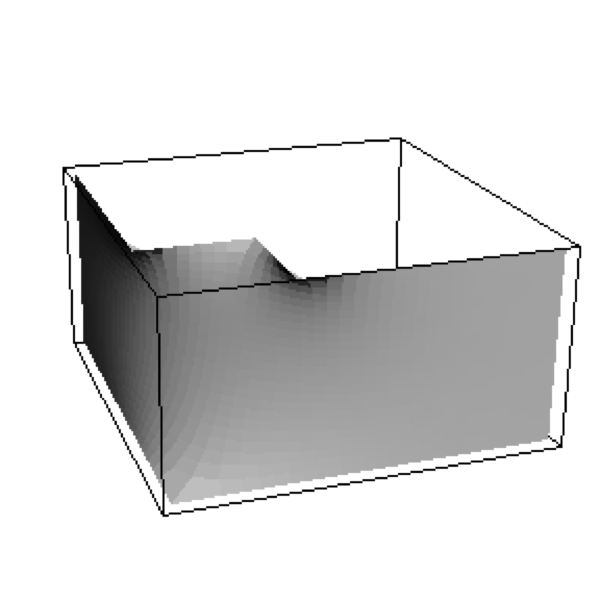} 
\caption{Conjugate surface patch of the doubly-periodic Scherk
surface with handles}
\label{fig:conjsch1}
\end{figure}

The surface patch will be bounded by two polygonal arcs in space.
The first one corresponds to the $P_1\ldots P_{2g}$ polygonal 
arc of the orthodisks.
As the original surface is cut orthogonally by two symmetry planes along this arc, the corresponding 
arc of the conjugate surface
will consist of orthogonal line segments which stay in the same horizontal plane:
to see this, begin by observing that the \ga map has vertical normals at 
the points corresponding to the vertices of the orthodisk boundaries,
hence also at the corners of the straight line segments
on the conjugate surfaces.  Yet where two of these segments 
meet, the tangent plane is tangent to both of them, hence the normal is
in the unique direction normal to both; as the \ga map is
vertical, both segments must be horizontal. We conclude then 
that all of the straight segments must be horizontal, and hence this
connected component of the boundary must lie in a horizontal plane. 

Simlarly the second connected arc, this time made up of a pair of
infinite arcs, is also horizontal, thus parallel to the first boundary arc.

We next claim that the two horizontal
connected components of the boundary have 
(a pair of) parallel infinite
edges as ends which lie on the same (pair of) vertical planes.
To see this, consider the (pair of) cycles around the two
ends $E_1$ and $E_2$: we have constructed the \wei data
one-forms $Gdh$ and $\frac1G dh$ so that the coordinate
one-forms $(G- \frac1G)dh$ and $i(G+\frac1G)dh$ have purely
real periods, while the coordinate one-form $dh$ has a purely
imaginary period.  Thus for the conjugate surface, the 
coordinate
one-forms $(G- \frac1G)dh$ and $i(G+\frac1G)dh$ have purely
imaginary periods, while the coordinate one-form $dh$ has a purely
real period. As any representative of this cycle lifts to connect 
semi-infinite ends of the connected components of the boundary,
and any such lift must have endpoints differing by a period,
we see that the semi-infinite ends of the connected components
of the boundary differ (respectively) by a purely vertical 
translation.

We now produce a minimal surface which spans this pair of boundary
components.  To do this, we first approximate the boundary by a compact
boundary, then solve the corresponding Plateau problem for compact boundary
values, and then finally take a limit.  More precisely, consider a 
boundary formed from the boundary arcs described above by introducing vertical
segments $\G_{b,1}$ and $\G_{b,2}$ connecting the two pairs of parallel semi-infinite
boundary edges at distance $b$ from the (image of the) point $P_g$. This boundary is now
compact and projects  injectively onto a rectangle boundary by using the
projection in the direction of the 'diagonal' vector $(1,1,0)$.

This corresponding boundary
problem has a classical Plateau solution which is
unique and a graph over the plane orthogonal to $(1,1,0)$ by Rad\'o's
theorem. We now look at a sequence of such Plateau
solutions for increasing values of
$b\to\infty$. Using solutions for smaller values of $b$ as barriers
for the solutions corresponding to the larger values of $b$, by
the maximum principle we see that the solutions $S_{g,b}^*$ form an increasing
sequence of graphs.
To show that this sequence actually converges, it suffices to 
show that the intersections $A_b$ of the family of surfaces with the 
vertical plane $\Pi$ passing through $P_g$ and $V_{\pm}$ lie in a single
compact set.  To see this, we consider a pair of (quarters of)
Scherk surfaces (conjugate to $S_0$), each of whose boundary 
components consist of an infinite
$L$ and a vertical translate of that $L$ by the same amount as for
a conjugate surface for $S_g$. The first such conjugate Scherk
passe through the point $P_g$, while the second is displaced to pass
through the points $V_{\pm}$.  This pair of boundary arcs lie to the
other side of the boundary arcs of $S_g$ in the respective horizontal
planes. Thus, this pair of surfaces meets the plane $\Pi$ in a pair of arcs,
which, by the maximum principle lie to either side of the arcs $A_b$
on the strip of $\Pi$ between our fixed pair of horizontal planes.

Thus, by Harnack, the approximate solutions $S_{g,b}^*$ converge
to a solution $S_g^*$ for the infinite boundary problem; as the approximate 
solutions are all graphs and minimal, the
limit $S_g^*$ is also a graph (here convergence of the
graphed functions $u_b$ in $C^0$ implies their convergence
in $C^1$ by standard elliptic theory, hence to a graph). Then,
by Krust's theorem, the conjugate patch to that limit
graph is also a graph, and since that 
conjugate patch is a fundamental piece of our surface $S_g$, we see
that $S_g$ is embedded.
\end{proof}

\addcontentsline{toc}{section}{References}
\bibliographystyle{plain}
\bibliography{bill}

\label{sec:liter}

\end{document}